\documentclass{article}
\usepackage{amsmath,amsfonts,amssymb,amsopn,amscd,amsthm,accents}
\usepackage{bbm,wasysym,stmaryrd}
\usepackage{hyperref,color}
\usepackage[normalem]{ulem}
\usepackage{graphicx}
\usepackage[nocompress]{cite}

\renewcommand{\P}{\mathbbm{P}}

\newcommand{\Exp}{\mathbbm{E}}
\newcommand{\E}{\mathcal{E}}
\DeclareMathOperator{\eqlaw}{\stackrel{\mathcal{L}}{=}}

\newcommand{\R}{\mathbbm{R}}
\newcommand{\N}{\mathbbm{N}}

\newcommand{\F}{\mathcal{F}}

\newcommand{\der}[2]{ \frac{\text{d} #1}{\text{d} #2} }  

\newcommand{\M}{\mathcal{M}}
\newcommand{\C}{\mathcal{C}}

\newcommand{\smax}{\lVert \sigma \rVert_{\infty}}
\newcommand{\ls}{\lambda_{*}}

\newcommand{\Jmax}{\lVert \bar{J} \rVert_{\infty}}
\newcommand{\bmax}{\lVert b \rVert_{\infty}}

\theoremstyle{plain}
\newtheorem{theorem}{Theorem}

\newtheorem{proposition}[theorem]{Proposition}
\newtheorem{lemma}[theorem]{Lemma}
\newtheorem{definition}[theorem]{Definition}

\newtheorem{remark}{\it Remark\/}

\begin{document}
\title{Large deviations for randomly connected neural networks: I. Spatially extended systems}
\date{\today}
\author{Tanguy Cabana \and Jonathan Touboul}
\maketitle

\begin{abstract}
In a series of two papers, we investigate the large deviations and asymptotic behavior of stochastic models of brain neural networks with random interaction coefficients. In this first paper, we take into account the spatial structure of the brain and consider (i) the presence of interaction delays that depend on the distance between cells and (ii) Gaussian random interaction amplitude whose mean and variance depend on the neurons positions and scale as the inverse of the network size. We show that the empirical measure satisfies a large-deviation principle with good rate function reaching its minimum at a unique spatially extended probability measure. This result implies averaged convergence of the empirical measure and propagation of chaos. The limit is characterized through complex non-Markovian implicit equation in which the network interaction term is replaced by a non-local Gaussian process whose statistics depend on the solution over the whole neural field.
\end{abstract}

\medskip

\tableofcontents

\medskip

\section{General Introduction}
In this two-part paper, we study the asymptotic behavior of neuronal networks with heterogeneous interconnections inspired by neurobiology. The model we consider describes the stochastic dynamics of $N$ interconnected neurons in random environment. Each neuron is described by a real${}^{{1}}$\footnote{${}^{{1}}$ The results generalize to multi-dimensional variables with obvious modifications.} state variable (e.g., its voltage) and the network satisfies the following stochastic differential equation (SDE):
\begin{equation}\label{eq:Standard equation}
	dX^{i,N}_t=\left(f(r_i,t,X^{i,N}_t) + \sum_{j=1}^{N} J_{ij} b(X^{i,N}_t, X^{j,N}_{t-\tau(r_i,r_j)}) \right)\,dt+\lambda(r) dW^{i}_t,
\end{equation}
	where $X^{i,N}_{t}$ is the state of neuron $i\in \{1,\cdots,N\}$ at time $t$ and:
\begin{itemize}
	\item the map $f$ denotes each neuron's intrinsic dynamics, which may depend on time  (e.g., to model non constant stimuli) and on the neuron identity through an heterogeneity parameter $r_{i}\in \R^{d}$. This parameter allows considering neurons with distinct properties, here associated to their location, and more general in the companion paper~\cite{cabana-touboul:16}.
	\item the terms $ (J_{{ij}}b(x^{i},x^{j}))$ represent the impact of the neuron $j$ with state $x^{j}$ onto the neuron $i$ with state $x^i$. The amplitude of these interaction coefficients $(J_{ij})_{(i,j) \in \{1,\cdots,N\}^{2}}$, called \emph{synaptic efficiencies}, are modeled as independent real Gaussian variables with mean and variance scaling as the inverse of the network size. Moreover, the time necessary to propagate and transmit the information is taken into account in the delay map $\tau$ depending on cells locations;
	\item the stochastic fluctuations are driven by independent standard real-valued Brownian motions at the level of each cell $(W^i_t)_{t\geq 0}$, with positive diffusion coefficient depending on the cell location.
\end{itemize}
In this framework, the scaling of the variance of the coefficients $J_{ij}$ of order $1/N$ is much slower than the more classical scaling in $\mathcal{O}(1/N^{2})$~\cite{sznitman:89,den2008large,lucon2014mean}. This scaling has been shown to preserve a non-trivial contribution of the synaptic weights fluctuations in the thermodynamic limit, leading to rich behaviors depending on disorder levels. In particular, physics literature showed that, under this scaling assumption, the disorder level governs the glassy transition in spin glass systems~\cite{sherrington-kirkpatrick:75}, a chaotic transition in randomly connected neural networks~\cite{sompolinsky-crisanti-etal:88}, and a transition in the Kuramoto model whose nature is debated~\cite{daido:92,daido:00,stiller-radon:98,stiller-radon:00}. Ben Arous and Guionnet~\cite{ben-arous-guionnet:95,guionnet:97,ben-arous-guionnet:98} developed a general mathematical methodology relying on large deviations techniques to characterize the thermodynamic limit of Langevin spin glass systems with bounded spins, linear interactions $b(x,y)=y$, and random connectivity coefficients with variance $1/N$. These methods were soon adapted to biological neural networks with sigmoidal interactions $b(x,y)=S(y)$ either in discrete time ~\cite{moynot-samuelides:02,dauce-moynot-etal:01,cessac-samuelides:07,faugeras2014asymptotic}, or in continuous time with multiple populations and delays~\cite{cabana-touboul:12}. All these cases did not take into account (i) the spatial structure of the network along with biologically relevant space-dependent delay, nor (ii) interactions depending on the state of both neurons (i.e., $b$ having a non-trivial dependence in its both variables). Addressing these two issues are precisely the topic of this two-parts paper.

This first paper addresses the specific difficulties associated with the spatial extension of the network keeping $b(x,y)=S(y)$, while the companion paper~\cite{cabana-touboul:16} deals with the case of a general interaction term, and will ignore spatial dependence. Our proof follows the general methodology developed in~\cite{ben-arous-guionnet:95}; the two papers are organized in a similar fashion and, to avoid any repetition, the demonstrations that do not present additional difficulties in the second part will only appear in the first. Furthermore, treating these two independent difficulties separately notably allows to obtain results valid for all times in the present paper, while specific difficulties of the companion paper~\cite{cabana-touboul:16} will lead us to restrict our analysis to a finite time interval, but in which we will prove stronger estimates.

The biological motivation for this first part is to strengthen the link between the microscopic dynamics of equation \eqref{eq:Standard equation} with the spatiotemporal patterns of activity recorded at the surface of the cortex. These patterns are biologically relevant for their correlate with cognitive processes such as memory~\cite{funahashi:89}, visual illusions~\cite{jancke:04} or the propagation of a localized stimulus~\cite{muller:14}. Historically, an efficient - though heuristic - mathematical model was famously provided by H. Wilson and J. Cowan~\cite{wilson-cowan:72,wilson-cowan:73} in order to describe the macroscopic dynamics of neural fields and account for these spatial phenomena. They proposed an integro-differential equation involving the level of activity $u(r,t)$ of cells at location $r$: 
\begin{equation}\label{eq:WC}
	\frac{\partial u}{\partial t} = -u(r,t) + \int_{D} J(r,r') S(u(r',t))\,dr' + I(r,t)
\end{equation}
where $I(r,t)$ represents the input at location $r$, $J(r,r')$ is the averaged interconnection weight from neurons at location $r'$ onto neurons at location $r$ and the non-decreasing map $S$ associates to a level of activity $u$ the resulting firing rate. While this equation has been very successful in reproducing a number of biological phenomena - in particular working memory~\cite{kilpatrick:13} and visual hallucination patterns~\cite{ermentrout:79,bressloff-cowan-etal:01} - it lacked a rigorous derivation. Justifying the relationship between the macroscopic dynamics of neural fields and the properties of finite-sized networks is a prominent issue in computational neuroscience~\cite{bressloff:12}. It has a been the topic of numerous researches using PDE formalisms and kinetic equations~\cite{cai-tao:04,rangan-cai:07,rangan-kovacic:08} with deep applications to the visual system, moment reductions and master equations~\cite{ly-tranchina:07,bressloff:09}, but also the development of specific Markov chain models reproducing in the thermodynamics limit the dynamics of Wilson-Cowan systems~\cite{buice-cowan:07,buice-cowan:10,bressloff:09,buice2013beyond,buice2013dynamic}. Spatially extended networks with connectivity coefficients having a variance scaling as $1/N^{2}$ were investigated in~\cite{touboul2014propagation,touboul2014spatially,lucon2014mean} using coupling or compactness methods. Eventually, we refer to the general work on large deviations for interacting heterogeneous diffusions with non-random interconnections~\cite{dai1996mckean}, a landmark work with numerous applications. Unfortunately, none of these techniques can handle cases where the synaptic couplings have a variance decaying in $1/N$ that we address here. This scaling has the advantage of enriching the original Wilson-Cowan equation by making the effect of heterogeneous connections explicit.

The organization of the first part is as follows. We provide in section~\ref{sec:MathSetting} the general notations and assumptions used here and in the companion paper~\cite{cabana-touboul:16}. In section~\ref{sec:MainResults}, we summarize the main results proved in this first manuscript. Sections~\ref{sec:LDP} is dedicated to the demonstration of a weak large deviations principle for the averaged system. In section~\ref{sec:LimitIdentification}, we develop an alternative and more compact methodology compared to~\cite{ben-arous-guionnet:95,guionnet:97} for characterizing the limit under our regularity assumptions. Moreover, we combine all results to conclude on the convergence of the empirical measure and propagation of chaos property. 

\section{Mathematical setting}\label{sec:MathSetting}
In order to take into account the random spatial organization of the network, we consider that the neurons have independent locations $r_{i}\in D$ drawn according to a probability measure $\pi \in \M_1^+(D)$ representing the density of neurons over the cortex, and assumed to be absolutely continuous with respect to Lebsegue's measure. Moreover, the connectivity coefficients $J_{ij}$ are Gaussian variables, independent conditionally on the neurons location, and with laws depending on the location of both neurons${}^{1}$\footnote{${}^{1}$To fix ideas, we will consider $J_{ij}=\frac{J(r_{i},r_{j})}{N} + \frac{\sigma(r_{i},r_{j})}{\sqrt{N}} \xi_{ij}$ with $\xi_{ij}$ independent standard Gaussian random variables independent of the variables $(r_{i})_{i\in\{1\cdots N\}}$. }: 
\[J_{ij}\sim \mathcal{N}\left(\frac{J(r_{i},r_{j})}{N}, \frac{\sigma^{2}(r_{i},r_{j})}{N}\right).\]
There are thus two sources of randomness:
\begin{itemize}
\item \emph{Random environment:} the locations of neurons and connectivity coefficients constitute a frozen environment; they are random variables in a probability space $(\tilde{\Omega}, \tilde{\mathcal{F}}, \mathcal{P})$. 
\item \emph{Stochastic dynamics:} the neurons are driven by a collection of independent $(\Omega, \mathcal{F},(\mathcal{F}_t), \P)$-Brownian motions $(W^i_t)_{i\in \mathbbm{N}}$ independent of the environment.
\end{itemize}

Hence, the dynamics of the $X^{i,N}$ depends both on the random environment (i.e., the realization of locations $\mathbf{r}$ and weights $J$) and noise (the realization of the Brownian motions). We will denote by $\E$ the expectation over the environment (i.e. with respect to the probability distribution $\mathcal{P}$) and introduce the shorthand notation $\mathcal{P}_J$ and $\mathcal{E}_J$ the probability and expectation over the synaptic weights matrix $J$ only (that is, $\mathcal{P}$ and $\E$ conditioned over the positions $\mathbf{r}$). 

We make the following regularity assumptions:
\begin{enumerate}
	\item The function $f: D \times \R \times \R \mapsto \R$ is $K_f$-Lipschitz continuous in its three variables.
	\item The map $b:\R \times \R \mapsto \R$ is bounded ($\bmax:=\sup_{x,y \in {\R \times \R}} \vert b(x,y) \vert < \infty$) and $K_b$-Lipschitz-continuous in all variables.
  	\item The mean and variance maps of the weights $J$ and $\sigma$ are bounded and, respectively, $K_J$ and $K_{\sigma}$-Lipschitz continuous in their second variable. We denote
  	\[
  	\lVert J \rVert_{\infty}=\underset{(r,r')\in D^2}{\sup} |J(r,r')|, \quad
  	\lVert \sigma \rVert_{\infty}=\underset{(r,r')\in D^2}{\sup} \sigma(r,r').
  	\]
    \item $\tau: D^2 \to \R^+$ is Lipschitz continuous, with constant $K_{\tau}$. It is in particular bounded, by compactness of $D$. We denote by $\bar{\tau}$ its supremum.
    \item The diffusion coefficient $\lambda: D \to \R_+^*$ is a $K_{\lambda}$ Lipschitz continuous and uniformly lower-bounded: $\forall r \in D$, $\lambda(r)\geq \ls >0$.
\end{enumerate}

Throughout the paper, $\M_1^+(\Sigma)$ will denote the set of Borel probability measure on the Polish space $\Sigma$. Let $\C_{\tau}:=\C([-\bar{\tau},0], \R)$, and $\mu_0:D \to \M_1^+\big( \C_{\tau} \big)$ taking value in the space of probability distributions with bounded second moment, and which is continuous in the variable $r \in D$ in the sense that there exists a random mapping $\bar{x}^0: D \to \C_{\tau}$ on $\big(\Omega, \mathcal{F}, \P\big)$ and $C_0 >0$ such that:
\begin{equation}\label{hyp:spaceRegInitCond}
\forall r, r' \in D, \; \mathcal{L}(\bar{x}^0(r))=\mu_0(r), \quad \Exp \Big[ \underset{-\bar{\tau} \leq s \leq 0}{\sup} \big|\bar{x}^0_s(r) -\bar{x}^0_s(r') \big|^2 \Big] \leq C_0 \Vert r-r' \Vert_{\R^d}^2,
\end{equation}
where $\Vert \cdot \Vert_{\R^d}$ denote the Euclidean norm on $\R^d$. We denote by $\Vert \cdot \Vert_{\tau, \infty}$ the supremum norm on $\C_{\tau}$. We consider that the network's initial conditions are independent realizations of $\mu_0$:
\begin{equation}\label{eq:ICNet}
	\text{Law of } (x_t)_{t\in[-\bar{\tau},0]} = \bigotimes_{i=1}^N \mu_0(r_i).
\end{equation}

Since equation~\eqref{eq:Standard equation} is a standard delayed stochastic differential equation with regular parameters, is is classically well-posed~\cite{da-prato:92,mao:08}:

\begin{proposition}\label{pro:ExistenceUniquenessNetwork}
	For each $\mathbf{r}\in D^N$, and $J \in \R^{N\times N}$ and $T>0$, there exists a unique weak solution to the system~\eqref{eq:Standard equation} defined on $[-\bar{\tau}, T]$ with initial condition~\eqref{eq:ICNet}. Moreover, this solution is square integrable.
\end{proposition}
\begin{remark}
	Note that if the initial condition was given by $(X^{i,N}_t)_{t\in[-\bar{\tau},0]}=\zeta^i$ with $\zeta^i\eqlaw \mu_0(r_i)$, we can of course prove strong existence and uniqueness of solutions.
\end{remark}

We now work with an arbitrary fixed time $T>0$ and denote by $Q^N_{\mathbf{r}}(J)$ the unique law solution of the network equations restricted to the $\sigma$-algebra $\sigma(X^{i,N}_s, 1\leq i\leq N, -\bar{\tau}\leq s\leq T)$. $Q^N_{\mathbf{r}}(J)$ is a probability measure on $\C^N$ where $\C:= \mathcal{C}\big([-\bar{\tau},T],\R\big)$ is the space of real valued continuous functions from $[-\bar{\tau},T]$ to $\R$. This measure depends on the realizations of both the connectivity matrix $J$ and the locations of neurons $\mathbf{r}$. For any $t \in [0,T]$, we will denote by $\Vert \cdot \Vert_{\infty,t}$ the supremum norm on $\C \big([-\bar{\tau},t], \R \big)$. In order to characterize the behavior of the system as the network size diverges, we will show a Large Deviations Principle (LDP) for the double-layer empirical measure $\hat{{\mu}}_{N}\in \mathcal{M}_{1}^{+}(\C\times D)$ defined as:
\begin{equation}\label{eq:EmpiricalLaw}
\hat{\mu}_N := \frac{1}{N} \sum_{i=1}^N \delta_{(X^{i,N},r_i)},
\end{equation}
where $\delta_{(x,r)} \in \M_1^+(\C \times D)$ denotes the degenerate probability measure at $(x,r) \in \C \times D$. 

We will rely on  Sanov's theorem which is central in large deviations techniques. This result states that the empirical measure of independent copies of the same law $\mu$ on a Polish space $\Sigma$ satisfies a full LDP with good rate function corresponding to the relative entropy $I(.|\mu)$ defined, for $\nu \in \M_1^+(\Sigma)$, by:
\begin{equation*}
I(\nu|\mu) :=
\begin{cases}
  \displaystyle{\int \log{ \frac{d\nu}{d\mu} } }d\nu  &  \text{if }  \nu \ll \mu,\\
  + \infty & \text{otherwise }.
\end{cases}
\end{equation*}

Because of the interaction term and the lack of symmetry due to the particular realization of the coupling coefficients, it is clear that the variables $(X^{i,N},r_{i})$ satisfying equation~\eqref{eq:Standard equation} are neither independent nor identically distributed, and one cannot rely directly on Sanov's theorem to show an LDP for the empirical measure~\eqref{eq:EmpiricalLaw}. Sanov's theorem will nevertheless apply in an averaged sense for a network with no interaction, and we will build upon this result and on a particular generalization of Varadhan's lemma~\cite{dembo-zeitouni:09,deuschel-stroock:89} to derive a weak LDP for the empirical measure of the original network~\eqref{eq:Standard equation} averaged on the environment variables. 

The law of a neuron at location $r \in D$ in the absence of interaction, denoted $P_r\in \M_{1}^{+}(\C)$, is the unique weak solution of the one-dimensional standard SDE:
\begin{equation}\label{eq:Uncoupled}
	\begin{cases}
			dX_t= f(r,t,X_t) dt + \lambda(r) dW_t\\
			(X_t)_{t \in [-\bar{\tau},0] } \eqlaw \mu_0(r).
	\end{cases}
\end{equation}
Remark that, by a direct application of Girsanov's theorem, $Q^N_{\mathbf{r}}(J)$ is absolutely continuous with respect to $P_{\mathbf{r}}:=\bigotimes_{i=1}^N P_{r_i}$, with density:
\begin{multline}\label{eq:Density}
  \der{Q^N_{\mathbf{r}}(J)}{P_{\mathbf{r}}}(\mathbf{x}) = \exp\Bigg(\sum_{i=1}^N  \int_{0}^{T}  \Big(\frac{1}{\lambda(r_i)}  \sum_{j=1}^N   J_{ij} b(x^{i}_t, x^{j}_{t-\tau(r_i,r_j)})\Big)   dW_t(x^i,r_i)\\
  - \frac{1}{2} \int_{0}^{T} \Big(\frac{1}{\lambda(r_i)} \sum_{j=1}^N J_{ij} b(x^{i}_t, x^{j}_{t-\tau(r_i,r_j)})\Big)^2 dt\Bigg),
\end{multline}
where $\mathbf x=(x^{1},\cdots, x^{N}) \in \C^{N}$ and 
\begin{equation}\label{def:W}
\forall (x,r) \in \C \times D, \; t \in [0,T], \qquad W_t(x,r):= \frac{x_t-x_0}{\lambda(r)} - \int_0^t \frac{f(r,s,x_s)}{\lambda(r)} ds.
\end{equation}
Remark that, by \eqref{eq:Uncoupled}, $\big(W_t(.,r)\big)_t$ is a $P_r$-Brownian motion. Moreover, under $P_{\mathbf{r}}$ the Brownian motions $\big(W_t(x^i,r_i), 0 \leq t \leq T \big)_{i \in \{ 1\cdots N\}}$ are independent. 

Under $P_{\mathbf{r}}$ neurons are independent but not identically distributed, while the pairs $(X^{i,N},r_i)$ are if we take into account the random character of the locations. We thus introduce the well-defined two-layer probability measure $P\in \M_1^+(\C \times D)$ of these variables, $dP(x,r):=dP_r(x) d\pi(r)$ (see Appendix~\ref{sec:Appendix.A} for details). Remark that Sanov's theorem does apply to the empirical measure~\eqref{eq:EmpiricalLaw} considered under $P^{\otimes N}$.

In order to obtain an analogous result for the full network, we construct its associated symmetric law:

\begin{lemma}\label{lemma:ContinuityOfSolutions}
Let $Q^N_{\mathbf{r}}:=\mathcal{E}_J \big( Q^N_{\mathbf{r}}(J) \big)$. Then the map
\begin{equation*}
\mathcal{Q}:
\left\{
\begin{array}{ll}
  D^N \to \mathcal{M}_1^+(\C^N)\\
  \mathbf{r} \to Q^N_{\mathbf{r}}

\end{array}
\right.
\end{equation*}
is continuous with respect to the weak topology. Moreover,
\[
dQ^N(\mathbf{x},\mathbf{r}):=dQ^N_{\mathbf{r}}(\mathbf{x})d\pi^{\otimes N}(\mathbf{r})
\]
defines a probability measure on $\mathcal{M}_1^+\big((\C \times D)^N \big)$.
\end{lemma}
This result is proved in Appendix~\ref{sec:Appendix.A}.

As announced in the introduction, we will adopt the classical firing-rate formalism~\cite{wilson-cowan:72,wilson-cowan:73,amari:72} in this first part that $b(x,y)=S(y)$. Here, $S$ is a sigmoidal transformation assumed to be a smooth (at least continuously differentiable) increasing map tending to $0$ at $-\infty$ and to $1$ at $\infty$. 

Equation \eqref{eq:Standard equation} thus becomes:
\begin{equation}\label{eq:Network}
	dX^{i,N}_t=\left(f(r_i,t,X^{i,N}_t) + \sum_{j=1}^{N} J_{ij} S(X^{j,N}_{t-\tau(r_i,r_j)}) \right)\,dt+\lambda(r) dW^{i}_t.
\end{equation}

We now state the main results for this dynamics.

\section{Statement of the results}\label{sec:MainResults}

\begin{theorem}\label{thm:Convergence}
	$\Big(Q^N\big(\hat{\mu}_{N} \in \cdot \big) \Big)_{N \in \N^*}$ satisfies a weak LDP of speed $N$ and converges towards $\delta_{Q} \in \M_1^+\Big( \M_1^+\big(\C \times D \big) \Big)$ as $N$ goes to infinity.
\end{theorem}

\begin{remark}
Note that, for $T < \frac{\ls^2}{2 \smax^2 \bmax^2}$, a full large deviation principle can be demonstrated, as we can readily prove exponential tightness of the averaged empirical measure (see e.g.~\cite{ben-arous-guionnet:95} and the companion paper~\cite{cabana-touboul:16}). As we are mainly interested in proving convergence of the empirical measure, a weak LDP will be sufficient in this first part, allowing showing convergence results in the absence of limitation in time. 
\end{remark}

The quantitative estimates leading to this convergence result are summarized in the following two results:
\begin{theorem}\label{thm:LDP}
  There exists a good rate function $H$ on $\M_1^+ (\C \times D)$ such that for any compact subset $K$ of $\M_1^+ (\C \times D)$,
  \[
  \limsup_{N\to\infty} \frac 1 N \log Q^N(\hat{\mu}_{N}\in K )\leq -\inf_{K} H.
  \]
\end{theorem}

The convergence result also relies on the tightness of the sequence of empirical measures:
\begin{theorem}\label{thm:Tightness}
  For any real number $\varepsilon>0$, there exists a compact subset $K_{\varepsilon}$ such that for any integer $N$,
  \[
  Q^{N}(\hat{\mu}_N \notin K_{\varepsilon})\leq \varepsilon.
  \]
\end{theorem}

These two results imply convergence of the empirical measure towards the set of minima of the rate function $H$. Their uniqueness and characterization is subject of the following theorem demonstrated in section~\ref{sec:LimitIdentification}:
\begin{theorem}\label{thm:Limit}
  The good rate function $H$ achieves its minimal value at the unique probability measure $Q \in \M_1^+ (\C \times D)$ satisfying:
    \[Q \simeq P, \qquad \der{Q}{P}(x,r)=\E \left[ \exp\bigg\{ \frac 1 {\lambda(r)}\int_0^T G_t^{Q}(r) dW_t(x,r) - \frac 1 {2\lambda(r)^2} \int_0^T {G_t^{Q}(r)}^2 dt \bigg\} \right]\]
    where $(W_t)_{t\in [0,T]}$ is a $P$-brownian motion, and $G^{Q}(r)$ is, under $\mathcal{P}$, a Gaussian process with mean and covariance
    \[
\begin{cases}
	\E [G^{Q}_t(r)] = \int_{\C \times D} J(r,r')  S(x_{t-\tau(r,r')})dQ(x,r') \\
	    \E [G^{Q}_t(r)G^{Q}_s(r)] = \int_{\C \times D} \sigma(r,r')^2 S(x_{t-\tau(r,r')})S(x_{s-\tau(r,r')})dQ(x,r).
\end{cases}
    \]
\end{theorem}

\medskip

The convergence result of Theorem~\ref{thm:Convergence} also implies propagation of chaos, thanks to a result due to Sznitman~\cite[Lemma 3.1]{sznitman:84}:
\begin{theorem}\label{thm:PropagationOfChaos}
	The system enjoys the propagation of chaos property. In other terms, $Q^N$ is $Q$-chaotic, i.e. for any bounded continuous functions $\phi_1,\cdots,\phi_m \in \C_b\big( \C \times D \big)$ and any neuron indexes $(k_1,\cdots, k_m)$, we have:
	\[
	\lim_{N\to\infty}\int \prod_{j=1}^m \phi_j(x^{k_j},r_{k_j})dQ^N(x,r) = \prod_{j=1}^m \int \phi_j(x,r)dQ(x,r).
	\]
\end{theorem}

We now proceed to the proof of our results.

\section{Large Deviation Principle}\label{sec:LDP}
The aim of this section is to establish the weak large deviation principle for our network. It relies on three key points. First, we will characterize the good rate function; the intuition for constructing this functional comes from Varadhan's lemma, but in the present case that lemma does not readily apply and we need to demonstrate that the candidate is indeed a good rate function. Second, we will show an upper-bound result on compact sets. The spatial framework will introduce new difficulties, requiring the introduction of an appropriate distance on $\C \times D$. Third, the tightness of our collection of empirical measures will allow to conclude on a weak large-deviations principle. 

\subsection{Construction of the good rate function}\label{sec:GoodRate}

Let us consider the interaction term of \eqref{eq:Network}:
\[
G^{i,N}_t(\mathbf{x},\mathbf{r}):=\frac{1}{\lambda(r_i)} \sum_{j=1}^N J_{ij} S(x^j_{t-\tau(r_i,r_j)}).
\]
By virtue of a heuristic central limit theorem, this term should behave as a Gaussian process in the large $N$ limit. With this in mind, we introduce, for $\mu \in {\M_1^+(\C \times D)}$, the plausible mean and covariance functions of its limit, respectively defined on $[0,T]^2 \times D$ and $[0,T] \times D$:
\begin{equation*}
\begin{cases}
	K_{\mu}(s,t,r)&:=\displaystyle{\frac{1}{\lambda(r)^2} \int_{\C \times D} \sigma(r,r')^2 S(x_{t-\tau(r,r')})S(x_{s-\tau(r,r')}) d\mu(x,r')}\\
	m_{\mu}(t,r) &:= \displaystyle{\frac{1}{\lambda(r)} \int_{\C \times D} J(r,r') S(x_{t-\tau(r,r')})d\mu(x,r')}.\\
\end{cases}
\end{equation*}
Here, $\mu$ can be understood as the putative limit law of the couple $(x^j,r_j)$ if it exists. Covariance and mean functions $K_{\mu}$ and $m_{\mu}$ are well defined as for every fixed $r \in D$ the two maps
\[
A_r:(x,r')\to \frac{1}{\lambda(r)} J(r,r') S(x_{\cdot-\tau(r,r')}), \quad
\tilde{A}_r(x,r')\to \frac{1}{\lambda(r)^2} \sigma(r,r')^2 S(x_{\cdot-\tau(r,r')})S(x_{\cdot-\tau(r,r')})
\]
are continuous
for the classical product norm $\lVert (x,r') \rVert_{\C \times D} := \underset{t\in [-\bar{\tau},T]}{\sup}|x(t)|+ \lVert r \rVert_{\R^d}$. Hence, they are Borel-measurable, and integrable with respect to every element of $\M^+_1(\C \times D)$.
Remark that, since $S$ takes value in $[0,1]$, both functions are bounded: $|K_{\mu}(s,t,r)| \leq \frac{\lVert \sigma \rVert_{\infty}^2}{\ls^2}$ and $|m_{\mu}(t,r)| \leq \frac{\lVert J \rVert_{\infty} } {\ls}$. Moreover, as $\mu$ charges continuous functions, $K_{\mu}$ and $m_{\mu}$ are continuous maps by the dominated convergence theorem. 

Clearly enough, $K_{\mu}$ has a covariance structure. As a consequence, we can define a probability space $(\hat{\Omega},\hat{\F},\gamma)$ and a family of independent stochastic processes $\big(G^{\mu}(r)\big)_{\mu \in \M_1^+(\C \times D),r \in D}$ for any measure $\mu \in \M^+_1(\C \times D)$, such that $G^{\mu}(r)$ is a centered Gaussian process with covariance $K_{\mu}(.,.,r)$ under $\gamma$. This ensures continuity of the map $r \to \mathcal{L}\big( G^{\mu}(r)\big)$. We will denote by $\E_{\gamma}$ the expectation under this measure.

\begin{remark}\label{muDependence}
\begin{enumerate}
\item As in \cite{cabana-touboul:12,ben-arous-guionnet:95}, we could alternatively have defined a family of probability measure $\big(\gamma_{\mu}\big)_{\mu \in \M^+_1(\C \times D)}$, and a family of Gaussian processes $\big( G(r) \big)_{r \in D}$ with covariance $K_{\mu}$ under $\gamma_{\mu}$. This approach is equivalent to ours, but the latter present the advantage of being very adapted to Fubini's theorem.
\item The family of processes $\big(G^{\mu}_t(r)\big)_{\mu,r}$ is intended to encompass possible candidates for the effective asymptotic interactions $\lim_N \Big(G^{i,N}(\mathbf{x},\mathbf{r})\Big)_{i \in \N^*}$. In these interactions, the Gaussian weights are independent for different particles, so that it seems natural to assume independence of $\big(G^{\mu}_t(r)\big)$ for different locations. Notably in our proof, we can swap from a continuous version of $G^{\mu}_t(r)$ to an independent one very easily, as they  are never taken jointly under $\gamma$. Thus, we can literally choose their covariance structure. For the sake of measurability under any Borel measure of $\M_1^+\big( \C \times D \big)$, we will mainly work with the continuous version of $G^{\mu}(r)$, and will explicitly introduce independent versions when independence is needed.
\end{enumerate}
\end{remark}	
	
We recall a few general properties on the relative entropy that are often used throughout the paper. For $p$ and $q$ two probability measures on a Polish space $E$ (see e.g.~\cite[Lemma 3.2.13]{deuschel-stroock:89}), we have the identity:
\begin{equation*}
I(q\vert p) = \sup \left\{\int_{E} \Phi dq - \log \int_{E} \exp \Phi dp \;;\; \Phi \in \C_b(E)\right\},
\end{equation*}
which implies in particular that for any bounded measurable function $\Phi$ on $E$,
\begin{equation}
  \int_{E} \Phi dq  \leq I(q\vert p) + \log \int_{E} \exp \Phi dp. \label{eq:IneqRelativeEntropy}
\end{equation}
If $\Phi$ is a lower-bounded (or upper-bounded) measurable function this inequality holds by monotone convergence.

We introduce some useful objects for the general scope of our demonstration. For any Gaussian process $(G_t)_{t \in [0,T]}$ of $\big(\hat{\Omega}, \hat{\mathcal{F}}, \gamma \big)$, and any $t \in [0,T]$
\begin{equation}\label{def:Lambda}
\Lambda_t(G):=\frac{ \exp\Big\{ - \frac{1}{2} \int_0^t G_s^2 ds \Big\}}{\E_{\gamma} \Big[\exp\Big\{ - \frac{1}{2} \int_0^t G_s^2 du\Big\} \Big]}.
\end{equation}
For any $t \in [0,T]$, $r \in D$, and $\nu \in \M_1^+(\C \times D)$ the following defines a probability measure on $\big(\hat{\Omega}, \hat{\mathcal{F}}\big)$ (see \cite{neveu:70}):
\[
d\gamma_{\widetilde{K}_{\nu,r}^{t}}(\omega) :=  \Lambda_t(G^{\nu}(\omega,r)) d\gamma(\omega), \quad \forall \omega \in \hat{\Omega}.
\]
As proven in \cite{neveu:70}, $G^{\nu}(r)$ is still a centered Gaussian process under $\gamma_{\widetilde{K}_{\nu,r}^{t}}$, with covariance given by:
\[
\widetilde{K}_{\nu,r}^{t}(s,u):=\E_{\gamma} \bigg[ G^{\nu}_u(r)G^{\nu}_s(r) \Lambda_t\big(G^{\nu}(r)\big) \bigg].
\]
We also define for any $\nu \in \M_1^+\big(\C \times D\big)$, $(x,r) \in \C \times D$ and $t \in [0,T]$, the processes
\begin{equation}\label{LnuVnuDef}
L^{\nu}_{t}(x,r) := \int_0^t G^{\nu}_s(r)\Big( dW_s(x,r) - m_{\nu}(s,r) ds\Big), \quad V^{\nu}_t(x,r):= W_t(x,r) - \int_0^t m_{\nu}(s,r) ds.
\end{equation}

Here are a few properties for these objects:

\begin{proposition}\label{LambdaProperties}
There exists a constant $C_T>0$, such that for any $\nu \in \M_1^+\big(\C \times D \big)$, $r \in D$, $t \in [0,T]$,

\begin{equation}\label{ineq:BoundOnLambda}
	\underset{0\leq s,u \leq t}{\sup} \widetilde{K}_{\nu,r}^{t}(s,u) \leq C_T,\quad 
	\Lambda_t\big(G^{\nu}(r) \big) \leq C_T,
	\end{equation}

	\begin{align}\label{eq:KtildeExpectation}
	\E_{\gamma} \bigg[ \exp\bigg\{-\frac{1}{2} \int_0^T {G^{\nu}_t(r)}^2 dt\bigg\} \bigg] = \exp\Big\{ -\frac{1}{2} \int_0^T \widetilde{K}_{\nu,r}^t(t,t) dt \Big\}.
	\end{align}
Moreover, if $(G_t)_{0\leq t \leq T}$ and $(G'_t)_{0\leq t \leq T}$ are two centered Gaussian processes of $\big(\hat{\Omega},\hat{\mathcal{F}},\gamma\big)$ with uniformly bounded covariance, then exists $\tilde{C}_T>0$ such that for all $t \in [0,T]$,
\begin{align}\label{ineq:DiffLambda}
\big| \Lambda_t(G)-\Lambda_t(G') \big| \leq \tilde{C}_T  \bigg\{ \int_0^t \E_{\gamma} \Big[ \big(G_s - {G'_s}\big)^2 \Big]^{\frac{1}{2}} ds + \int_0^t \big|G_s^2 - {G'_s}^2\big| ds \bigg\}.
\end{align}

\end{proposition}

\begin{proof}
Observe that by Jensen inequality:
\[
\Lambda_t \big( G^{\nu}(r) \big) \leq \E_{\gamma} \Big[\exp\Big\{ - \frac{1}{2} \int_0^t G^{\nu}_s(r)^2 du\Big\} \Big]^{-1} \overset{\text{Jensen}}{\leq} \exp\Big\{ \frac{1}{2} \int_0^t \E_{\gamma} \Big[G^{\nu}_s(r)^2\Big] du\Big\} \leq \exp\Big\{ \frac{\smax^2 t}{2\ls^2} \Big\}.
\]
As a consequence:
\begin{align*}
	\widetilde{K}_{\nu,r}^{t}(s,u) = \E_{\gamma} \Bigg[ G^{\nu}_u(r)G^{\nu}_s(r) \Lambda_t\big(G^{\nu}(r) \big) \Bigg] & \overset{\mathrm{C.S.}}{\leq}  \sqrt{K_{\nu}(s,s,r)K_{\nu}(t,t,r)} \exp\Big\{ \frac{\smax^2 t}{2\ls^2} \Big\} \\
	& \leq \frac{\smax^2}{\ls^2} \exp\Big\{ \frac{\smax^2 t}{2\ls^2} \Big\},
\end{align*}
For equality \eqref{eq:KtildeExpectation}, let $f(t):=\E_{\gamma} \bigg[ \exp\bigg\{-\frac{1}{2} \int_0^t {G^{\nu}_s(r)}^2 ds\bigg\} \bigg]$. As $(t,\omega) \to {G^{\nu}_t(\omega,r)}^2 \exp\bigg\{-\frac{1}{2} \int_0^t {G^{\nu}_s(\omega,r)}^2 ds\bigg\}$ is a well defined, $\gamma$-a.s. continuous, and integrable under $\gamma$, we have
	\begin{align*}
	f'(t) &= -\frac{1}{2}\E_{\gamma} \bigg[{G^{\nu}_t(r)}^2 \exp\bigg\{-\frac{1}{2} \int_0^t {G^{\nu}_s(r)}^2 ds\bigg\} \bigg] = -\frac{1}{2} \widetilde{K}_{\nu,r}^t(t,t) f(t),
	\end{align*}
	so that integrating $\frac{f'}{f}$ gives the result.
Furthermore, letting $(G_t)_{0\leq t \leq T}$ and $(G'_t)_{0\leq t \leq T}$ be two centered $\gamma$-Gaussian processes with variance bounded by a common constant $C_T$, we have:

\begin{align*}
& \big| \Lambda_t(G)-\Lambda_t(G')\big| = \bigg| \frac{\exp\Big\{ - \frac{1}{2} \int_0^t G_s^2 ds \Big\}}{\E_{\gamma} \Big[ \exp\Big\{ - \frac{1}{2} \int_0^t G_s^2 ds \Big\}\Big]}- \frac{\exp\Big\{ - \frac{1}{2} \int_0^t {G'_s}^2 ds \Big\}}{\E_{\gamma} \Big[ \exp\Big\{ - \frac{1}{2} \int_0^t {G'_s}^2 ds \Big\}\Big]} \bigg| \\
& \leq \exp\Big\{ \frac{\smax^2 t}{\ls^2} \Big\} \Bigg\{ \bigg| \E_{\gamma} \bigg[ \exp\Big\{ - \frac{1}{2} \int_0^t G_s^2 ds \Big\}-\exp\Big\{ - \frac{1}{2} \int_0^t {G'_s}^2 ds \Big\}\bigg] \bigg|\\
& + \bigg| \exp\Big\{ - \frac{1}{2} \int_0^t G_s^2 ds \Big\}-\exp\Big\{ - \frac{1}{2} \int_0^t {G'_s}^2 ds \Big\}\bigg| \Bigg\},\\
& \leq \frac{1}{2}  \exp\Big\{ \frac{\smax^2 t}{\ls^2} \Big\} \bigg\{ \int_0^t \E_{\gamma} \Big[ \big|G_s^2 - {G'_s}^2\big| \Big] ds + \int_0^t \big|G_s^2 - {G'_s}^2\big| ds \bigg\},
\end{align*}
where we have used the Lipschitz-continuity of exponential on $\R_-$. Consequently, relying on Cauchy-Schwarz inequality, we obtain
\begin{align*}
\big| \Lambda_t(G)-\Lambda_t(G')\big| \overset{\mathrm{C.S.}}{\leq} \tilde{C}_T  \bigg\{ \int_0^t \E_{\gamma} \Big[ \big(G_s - {G'_s}\big)^2 \Big]^{\frac{1}{2}} ds + \int_0^t \big|G_s^2 - {G'_s}^2\big| ds \bigg\}.
\end{align*}
\end{proof}

We now state a key result to our analysis

\begin{lemma}\label{lemma2}

\begin{equation*}
\der{Q^N}{P^{\otimes N}} (\mathbf{x},\mathbf{r}) = \exp{\Big\{ N \bar{\Gamma}(\hat{\mu}_N) \Big\} },
\end{equation*}
where
\begin{equation*}
\bar{\Gamma}(\hat{\mu}_N) :=\frac{1}{N} \sum_{i=1}^N \log \E_{\gamma} \bigg[ \exp\bigg\{ \int_0^T \big(G^{\hat{\mu}_N}_t(r_i)+m_{\hat{\mu}_N}(t,r_i)\big)dW_t(x^i,r_i) - \frac{1}{2} \int_0^T \big(G^{\hat{\mu}_N}_t(r_i)+m_{\hat{\mu}_N}(t,r_i)\big)^2dt\bigg\}  \bigg],
\end{equation*}
with $W$ defined as in \eqref{def:W}.
\end{lemma}

\begin{proof}
Let, for all $(\mathbf{x},\mathbf{r}) \in (\C \times D)^N$
\[
X^{N}_i(\mathbf{x},\mathbf{r}) := \int_0^T G^{i,N}_t(\mathbf{x},\mathbf{r}) dW_t(x^i,r_i) -\frac{1}{2} \int_0^T {G^{i,N}_t(\mathbf{x},\mathbf{r})}^2 dt,
\]
which is well defined under $P_{\mathbf{r}}$ as $W(\cdot,r)$ is a $P_r$-Brownian motion.
Going back to equation \eqref{eq:Density}, we find:

\begin{equation*}
	\der{Q^N_{\mathbf{r}}(J)}{P_{\mathbf{r}}}(\mathbf{x}) = \exp\Big(\sum_{i=1}^N  X^{N}_i(\mathbf{x},\mathbf{r})\Big).
\end{equation*}

Averaging on $J$ and applying Fubini theorem, we find that $Q^N_{\mathbf{r}} \ll P_{\mathbf{r}}$, with density $\der{Q^N_{\mathbf{r}}}{P_{\mathbf{r}}}(\mathbf{x}) = \mathcal{E}_J \Big[ \exp\Big(\sum_{i=1}^N  X^{N}_i(\mathbf{x},\mathbf{r})\Big) \Big]$.
Moreover, equalities $dQ^N(\mathbf{x},\mathbf{r})=dQ^N_{\mathbf{r}}(\mathbf{x})d\pi^{\otimes N}(\mathbf{r})$ and $dP^{\otimes N}(\mathbf{x},\mathbf{r})= dP_{\mathbf{r}}(\mathbf{x}) d\pi^{\otimes N}(\mathbf{r})$ give
\begin{align*}
	\frac{dQ^N}{dP^{\otimes N}}(\mathbf{x},\mathbf{r}) & = \prod_{i=1}^N \mathcal{E}_J \big[ \exp\big( X^{N}_i(\mathbf{x},\mathbf{r})\big) \big] = \exp\bigg\{ \sum_{i=1}^N \log \E_{J} \Big[ \exp\big( X^{N}_i(\mathbf{x},\mathbf{r})\big) \Big] \bigg\},
\end{align*}

\noindent where we have used the independence of the synaptic weights $J_{ij}$. Note that here $\mathbf{x}$ are coordinates, thus independent of the $J_{ij}$, and the fact that $\bigg\{ G^{i,N}_t(\mathbf{x},\mathbf{r}) , 0 \leq t \leq T \bigg\}$ is, under $\mathcal{P}_J$, a Gaussian process with covariance $K_{\hat{\mu}_N}(t,s,r_i)$, and mean $m_{\hat{\mu}_N}(t,r_i)$.
\end{proof}

Varadhan's lemma motivates to consider the map:
\begin{equation*}
\int_{\C \times D} \log \E_{\gamma} \Big[ \exp\big\{ X^{\mu}(x,r) \big\}  \Big] d\mu(x,r),
\end{equation*}
with, for all $(x,r) \in \C \times D$ and $\mu \in \M_1^+(\C \times D)$,
\begin{equation}\label{def:Xmu}
X^{\mu}(x,r) := \int_0^T \big( G^{\mu}_t(r) + m_{\mu}(t,r) \big) dW_t(x,r) -\frac{1}{2} \int_0^T \big( G^{\mu}_t(r) + m_{\mu}(t,r) \big)^2 dt.
\end{equation}
The rigorous definition of that functional and few useful properties are subject of the following:
\begin{proposition}\label{prop:GammaBehaviour}
The map
\begin{align}\label{eq:Gamma}
\Gamma := \mu \in \M_1^+\big(\C \times D \big) \to
\left\{
\begin{array}{cl}
{\int_{\C \times D} \log \E_{\gamma} \Big[ \exp\big\{ X^{\mu}(x,r) \big\}  \Big] d\mu(x,r)} & \text{if } I(\mu|P) < \infty,\\
+ \infty & \text{otherwise }.
\end{array}
\right.
\end{align}
is well defined in $\R \cup \{ + \infty\}$, and satisfies
\begin{enumerate}
\item $\Gamma \leq I(\cdot|P)$,
\item $\exists \iota \in ]0, 1[, e \geq 0$, $|\Gamma(\mu)| \leq \iota I(\mu|P) + e$.
\end{enumerate}
\end{proposition}

\begin{proof}

The definition of $\Gamma$ is trivial when $I(\mu|P)=+\infty$. When $I(\mu|P)<+\infty$, we have $\mu \ll P$, and as $W(\cdot,r)$ is a $P_r$-Brownian motion, Girsanov's theorem ensures that the stochastic integral $\int_0^T \big(G^{\mu}_t(r)+m_{\mu}(t,r) \big)dW_t(x,r)$ is well defined $\gamma$-almost surely under $\mu$.\\
\noindent {\bf (1):}\\
Let $F_{\mu}:= \log \E_{\gamma} \Big[ \exp\big\{ X^{\mu}(x,r)\big\} \Big]$ denote the integrand in the formulation of $\Gamma$~\eqref{eq:Gamma}. It is measurable as a continuous function of the maps $(x,r) \to \big(K_{\mu}(t,s,r), 0 \leq t,s \leq T \big), \big(m_{\mu}(t,r), 0 \leq t \leq T \big), \big(W_t(x,r), 0 \leq t \leq T \big)$ that are continuous. Nevertheless, because of the contribution of the mean term $m_{\mu}$\footnote{see the expression of $\Gamma_2$ in Proposition~\ref{pro:DecompositionGamma}}, it is not bounded from below, as was the case in \cite{ben-arous-guionnet:95}. Let us prove that it is still $\mu$-integrable. In fact, for any $M>0$
\[
- F_{\mu}^{-}(x,r) \leq F_{\mu}^+(x,r)-F_{\mu}^-(x,r) = F_{\mu}(x,r) \leq \log\Big( \E_{\gamma} \big[ \exp\{X^{\mu}(x,r)\} \big] \vee M^{-1} \Big) =: F_{\mu,M}(x,r),
\]
where $F_{\mu}^+$ and $F_{\mu}^-$ respectively denote the positive and negative part of $F_{\mu}$.
As $F_{\mu}^-$ and $F_{\mu,M}$ are measurable and bounded from below, inequality \eqref{eq:IneqRelativeEntropy} applies. Let $\alpha \geq 1$. On the one hand
\begin{align}\label{ineq:GammaBehavior1}
\alpha \int_{\C \times D}& F_{\mu,M}(x,r) d\mu(x,r) \leq I(\mu|P)+ \log\bigg\{ \int_{\C \times D} \exp\Big\{\alpha F_{\mu,M}(x,r) \Big\}  dP(x,r) \bigg\} \nonumber \\
& \overset{\text{Jensen}}{\leq} I(\mu|P)+ \log\bigg\{ M^{-\alpha} +\int_{\C \times D} \E_{\gamma} \Big[ \exp\big\{ \alpha X^{\mu}(x,r)\big\} \Big] dP(x,r) \bigg\} \nonumber\\
& \overset{\text{Fubini}}{\leq} I(\mu|P)+ \log\bigg\{ M^{-\alpha} + \E_{\gamma} \bigg[ \int_{D} \int_{\C} \exp\Big\{\alpha X^{\mu}(x,r) \Big\} dP_r(x) d\pi(r) \bigg] \bigg\},
\end{align}
with the right-hand side of the two latter inequalities being possibly infinite. On the other hand,
\begin{align}\label{ineq:GammaBehavior2}
& \alpha \int_{\C \times D} F_{\mu}^-(x,r) d\mu(x,r) = \alpha \int_{\C \times D} \Big( - \log \E_{\gamma} \big[ \exp\big\{X^{\mu}(x,r)\big\} \big] \Big)^+ d\mu(x,r) \nonumber \\
& \overset{\text{Jensen}}{\leq} \alpha \int_{\C \times D} \Big( - \E_{\gamma} \big[ X^{\mu}(x,r) \big] \Big)^+ d\mu(x,r) = \int_{\C \times D} \Bigg( \E_{\gamma} \bigg[ - \int_0^T \big(G^{\mu}_t(r)+m_{\mu}(t,r)\big)dW_t(x,r) \nonumber\\
& -\frac{1}{2} \int_0^T \big(G^{\mu}_t(r)+m_{\mu}(t,r)\big)^2 dt \bigg] + \E_{\gamma} \bigg[\int_0^T \big(G^{\mu}_t(r)+m_{\mu}(t,r)\big)^2 dt \bigg]\Bigg)^+ d\mu(x,r) \nonumber\\
& \overset{\eqref{eq:IneqRelativeEntropy}}{\leq} I(\mu|P)+ \log\bigg\{ \int_{\C \times D} \exp\bigg\{ \alpha \bigg( \E_{\gamma}\big[X^{\mu}(x,r)\big] + T \frac{\Jmax^2 + \smax^2}{\lambda_*^2}\bigg)^+ \bigg\} dP(x,r) \bigg\} \nonumber\\
& \overset{\text{Jensen, Fubini}}{\leq} I(\mu|P) + \alpha C_T + \log\bigg\{ \E_{\gamma} \bigg[ \int_{D} \int_{\C} \exp\Big\{\alpha X^{\mu}(x,r) \Big\} dP_r(x) d\pi(r) \bigg] \bigg\}.
\end{align}

Moreover, $W(.,r)$ being a $P_r$-Brownian motion, the martingale property yields
\begin{equation}\label{ineq:GammaBehavior3}
\E_{\gamma} \bigg[ \int_{D} \int_{\C} \exp\Big\{\alpha X^{\mu}(x,r) \Big\} dP_r(x) d\pi(r) \bigg]\leq  \int_D \E_{\gamma} \bigg[ \exp\Big\{\frac{\alpha^2-\alpha}{2} \int_0^T \big(G^{\mu}_t(r) + m_{\mu}(t,r) \big)^2 dt \Big\}  \bigg] d\pi(r).
\end{equation}
Letting $\alpha=1$, we can see that $F_{\mu}$ is $\mu$-integrable, with
\begin{equation}\label{ineq:integrabilityGamma}
\int_{\C \times D} |F_{\mu}(x,r)| d\mu(x,r) = \int_{\C \times D} F_{\mu}^-(x,r)+F_{\mu,1}(x,r) d\mu(x,r) \leq 2 I(\mu|P)+ C_T+\log(2).
\end{equation}
Moreover,
\[
\Gamma(\mu) := \int_{\C \times D} F_{\mu}(x,r) d\mu(x,r) \overset{\eqref{ineq:GammaBehavior1}}{\leq} I(\mu|P) + \log\big\{ M^{-1}+1\big\},
\]
so that letting $M \to + \infty$ yields the result.

\noindent{\bf (2):}\\

For $\alpha \geq 1$, inequalities \eqref{ineq:GammaBehavior1}, \eqref{ineq:GammaBehavior2}, and \eqref{ineq:GammaBehavior3} ensure that
\begin{equation*}
\alpha |\Gamma(\mu)| \leq I(\mu|P)+ \alpha C_T + \bigg|\log\bigg\{\int_D \E_{\gamma} \bigg[ \exp\Big\{\frac{\alpha^2-\alpha}{2} \int_0^T \big(G^{\mu}_t(r) + m_{\mu}(t,r) \big)^2 dt \Big\}  \bigg] d\pi(r) \bigg\} \bigg|.
\end{equation*}

We recall that basic Gaussian calculus gives
\begin{equation*}\label{eq:GaussianExponentialQuadraticMoment}
\Exp \Big[ \exp\Big\{ \frac{1}{2} \mathcal{N}(m,v)^2 \Big\}\Big]=\frac{1}{\sqrt{1-v}}\exp\Big\{\frac{m^2}{2(1-v)} \Big\}=\exp\Big\{\frac{1}{2}\Big(\frac{m^2}{1-v} - \log(1-v) \Big) \Big\}
\end{equation*}
as soon as $v<1$. Jensen's inequality and Fubini theorem yield
\[
\E_{\gamma} \bigg[ \exp\Big\{\frac{(\alpha^2-\alpha)T}{2} \int_0^T \big(G^{\mu}_t(r) + m_{\mu}(t,r) \big)^2 \frac{dt}{T} \Big\} \bigg] \leq \int_0^T \E_{\gamma} \bigg[ \exp\Big\{\frac{(\alpha^2-\alpha)T}{2} \big(G^{\mu}_t(r) + m_{\mu}(t,r) \big)^2 \Big\}  \bigg] \frac{dt}{T}.
\]

As $\sqrt{(\alpha^2-\alpha)T}  \Big(G^{\mu}_t(r) + m_{\mu}(t,r) \Big) \sim \mathcal{N}\Big(\sqrt{(\alpha^2-\alpha)T} m_{\mu}(t,r) ,  (\alpha^2-\alpha)T K_{\mu}(t,t,r) \Big)$ under $\gamma$ then, for $(\alpha-1)$ small enough, exists a constant $C_T$ uniform in space such that
\begin{align*}
\E_{\gamma} \bigg[ \exp\Big\{\frac{(\alpha^2-\alpha)}{2} \int_0^T \big(G^{\mu}_t(r) + m_{\mu}(t,r) \big)^2 dt \Big\}  \bigg] & \leq \exp\Big\{(\alpha-1)C_T + \underbrace{o(\alpha-1)}_{\text{uniform in } r} \Big\}\\
& \leq \exp\Big\{(\alpha-1)C_T  \Big\}.
\end{align*}
This eventually yields
\begin{equation*}
|\Gamma(\mu)| \leq \iota I(\mu|P)+ e,
\end{equation*}
with $\iota:=\frac{1}{\alpha}$, and $e:=(2\alpha-1) C_T$.
\end{proof}

\begin{remark}
Remark that $\hat{\mu}_N \not\ll P$, $\Gamma(\hat{\mu}_N) = + \infty$ so that it is not equal to $\bar{\Gamma}(\hat{\mu}_N)$. These objects are different in nature, as the latter is random and must be considered under a proper probability measure on $\big(\C \times D \big)^N$, making sense of the stochastic integrals over the $\big(W_t(x^i,r_i), 0 \leq t \leq T \big)_{i \in \{ 1 \cdots N\}}$ (which are well defined under $P_{\mathbf{r}}$).
\end{remark}

As $\C \times D$ and $\M_1^+(\C\times D)$ are Polish spaces, and as the $(X^{i,N},r_i)$ are independent identically distributed random variables under $P^{\otimes N}$, Sanov's theorem ensures that the empirical measure satisfies, under this measure, a LDP with good rate function $I(.|P)$. Furthermore, if $\Gamma$ was bounded and continuous, Varadhan's lemma would, as a consequence of Lemma~\eqref{lemma2}, entail a full LDP under $Q^N$, with good rate function given by 

\begin{align*}
H (\mu) 
:=
\left\{
\begin{array}{cl}
I(\mu|P) - \Gamma(\mu) & \text{if } I(\mu|P) < \infty,\\
\infty & \text{otherwise }.
\end{array}
\right.
\end{align*}

At this point, it would be easy to conclude provided that $\Gamma$ presented a few regularity properties. Unfortunately, Varadhan's lemma assumptions fail here, as $\Gamma$ is neither continuous nor bounded from above. Obtaining a weak LDP as well as the convergence of the empirical measure requires to come back to the basics of large deviations theory.

Observe that $\Gamma$ is a nonlinear function of $\mu$, involving in particular an exponential term depending on the mean and covariance structure of the Gaussian process. In order to handle terms of this type, a key technique proposed by Ben Arous and Guionnet is to linearize this map by considering the terms in the exponential as depending on an additional variable $\nu \in \M_1^+(\C \times D)$~\cite{ben-arous-guionnet:95,guionnet:97}. In our case, this family of linearizations are given by the maps:
\begin{align*}
\Gamma_{\nu} := \mu \in \M_1^+\big(\C \times D \big) \to
\left\{
\begin{array}{cl}
{\int_{\C \times D} \log \E_{\gamma} \Big[ \exp\big\{ X^{\nu}(x,r) \big\}  \Big] d\mu(x,r)} & \text{if } I(\mu|P) < \infty,\\
+ \infty & \text{otherwise }.
\end{array}
\right.
\end{align*}
where $\mu, \nu \in \M_1^+(\C \times D)$.

\begin{remark}
Observe that Proposition~\ref{prop:GammaBehaviour} also applies to $\Gamma_{\nu}$ for every $\nu \in \M_1^+(\C \times D).$ Moreover, observe that $\Gamma(\mu)=\Gamma_{\mu}(\mu)$ for any $\mu \in \M_1^+(\C \times D)$.
\end{remark}

Moreover, defining
\[
\bar{\Gamma}_{\nu}(\delta_{(x,r)}):= \log \E_{\gamma} \Big[ \exp\big\{ X^{\nu}(x,r) \big\} \Big],\quad \quad \bar{\Gamma}_{\nu}(\hat{\mu}_N):= \frac{1}{N} \sum_{i=1}^N \log \E_{\gamma} \Big[ \exp\big\{ X^{\nu}(x^i,r_i)) \big\}  \Big],
\]
we note that $\bar{\Gamma}_{\nu}(\hat{\mu}_N)=\frac{1}{N} \sum_{i=1}^N \bar{\Gamma}_{\nu} (\delta_{(x^i,r_i)})$. Introducing $Q_{\nu} \in \M_1^+(\C \times D)$ by
\begin{equation}\label{def:Qnu}
dQ_{\nu}(x,r):= \exp\big\{ \bar{\Gamma}_{\nu}(\delta_{(x,r)}) \big\} dP(x,r)= \E_{\gamma} \Big[ \exp\big\{ X^{\nu}(x,r) \big\}  \Big] dP(x,r),
\end{equation}
we thus have
\[
dQ_{\nu}^{\otimes N}(\mathbf{x},\mathbf{r})=\exp\big\{ N \bar{\Gamma}_{\nu}(\hat{\mu}_N) \big\}dP^{\otimes N}(\mathbf{x},\mathbf{r}).
\]

This equality highlights the fact that, applying again Sanov's theorem, the empirical measure satisfies a full LDP under $\big(Q_{\nu} \big)^{\otimes N}$, with good rate function $I(.|Q_{\nu})$. On the other hand, Vardhan's lemma suggests that $\hat{\mu}_N$ satisfies, under the same measure, a LDP with rate function 
\begin{equation*}
H_{\nu} : \mu \to\\
\left\{
\begin{array}{ll}
I(\mu|P) - \Gamma_{\nu}(\mu)  & \text{if } I(\mu|P) < + \infty,\\
+ \infty & \text{otherwise. }
\end{array}
\right.
\end{equation*}
This is, for now, only a supposition, as its original counterpart $\Gamma_{\nu}$, is not bounded from above nor continuous, and as $\Gamma_{\nu}(\hat{\mu}_N)$ and $\bar{\Gamma}_{\nu}(\hat{\mu}_N)$ are not equal. Still, assuming the result is true, uniqueness of the good rate function would imply that $H_{\nu}$ equals $I(.|Q_{\nu})$. We shall justify the definition of $Q_{\nu}$, and proceed to the rigorous demonstration of the latter equality in Theorem~\ref{thm:Qnu}. 

We now introduce a very useful decomposition of $\Gamma_{\nu}$ based on Gaussian calculus (see \cite{ben-arous-guionnet:95,cabana-touboul:12,neveu:70}).
\begin{proposition} \label{pro:DecompositionGamma}
	For every $\nu \in \M_1^+(\C \times D)$, $\Gamma_{\nu}$ admits the following decomposition:
\begin{equation}\label{eq:DecompositionGamma}
\Gamma_{\nu}= \Gamma_{1,\nu} + \Gamma_{2,\nu},
\end{equation}
where $\forall \mu \in \M_1^+(\C \times D)$
\begin{align*}
\Gamma_{1,\nu}(\mu)& := -\frac{1}{2} \int_{\C \times D} \int_0^T \bigg\{ \widetilde{K}_{\nu,r}^t(t,t) +m_{\nu}(t,r)^2 \bigg\} dt d\mu(x,r),
\end{align*}
and
\begin{align*}
\Gamma_{2,\nu}(\mu) :=
\left\{
\begin{array}{cl}
{\frac{1}{2} \int_{\C \times D} \int_{\hat{\Omega}} L^{\nu}_T(x,r)^2 d\gamma_{\widetilde{K}_{\nu,r}^{T}} d\mu(x,r) + \int_{\C \times D} \int_0^T m_{\nu}(t,r)dW_t(x,r) d\mu(x,r)} & \text{if } I(\mu|P) < \infty,\\
+ \infty & \text{otherwise }.
\end{array}
\right.
\end{align*}
In particular, we have $\Gamma = \Gamma_1+\Gamma_2$ where $\Gamma_i(\mu):=\Gamma_{i,\mu}(\mu)$ for $i\in \{1,2\}$ and any $\mu \in \M_1^+(\C \times D)$.
\end{proposition}

This decomposition has the interest of splitting the difficulties: while the first term will be relatively easy to handle (see Proposition~\ref{prop:MeanAndVarRegularity}), the local martingale term will require finer estimates based on Gaussian calculus and a number of tools from stochastic calculus theory. It is also useful to state the following lemma, central for our analysis (see \cite[Lemma 5.15]{ben-arous-guionnet:95}):

\begin{lemma}\label{lemma5.15}
For any $r \in D$, the following equality holds $P_r$ almost surely
\begin{align}\label{eq:QnurDensity}
\E_{\gamma} \Big[ \exp\big\{ X^{\nu}(x,r) \big\}\Big]= \exp\bigg\{\int_0^T O_{\nu}(t,x,r)dW_t(x,r)-\frac{1}{2}\int_0^T O_{\nu}^2(t,x,r)dt\bigg\}
\end{align}
where 
\begin{align*}
O_{\nu}(t,x,r)& := \E_{\gamma} \bigg[ \Lambda_t\big( G^{\nu}(r) \big) G^{\nu}_t(r) L^{\nu}_t(x,r) \bigg]+ m_{\nu}(t,r).
\end{align*}
\end{lemma}

We are now ready to state one of the main result of the chapter which proves the intuitive equality between the two rate functions $I(\cdot|Q_{\nu})$ and $H_{\nu}$.

\begin{theorem}\label{thm:Qnu}
$Q_{\nu}$ is a well defined probability measure on $\M_1^+(\C \times D)$, and $H_{\nu}(\mu)=I(\mu|Q_{\nu})$. In particular $H_{\nu}$ is a good rate function.
\end{theorem}

\begin{proof}
We recall that
\[
dQ_{\nu}(x,r) = \E_{\gamma} \Big[ \exp\big\{ X^{\nu}(x,r) \big\}\Big] dP(x,r),
\]
and further define for any $r \in D$
\[
dQ_{\nu,r}(r) = \E_{\gamma} \Big[ \exp\big\{ X^{\nu}(x,r) \big\}\Big] dP_r(x).
\]
It is clear that $Q_{\nu}$ and $Q_{\nu,r}$ are positive measure of $\C \times D$ and $\C$ respectively. Moreover, Novikov's criterion ensures that 
\[
\Bigg(\exp\bigg\{ \int_0^t \big( G^{\nu}_s(r)+m_{\nu}(s,t) \big)dW_s(x,r) - \frac{1}{2} \int_0^t  \big( G^{\nu}_s(r)+m_{\nu}(s,t) \big)^2 ds \bigg\}\Bigg)_{0 \leq t\leq T}
\]
is a uniformly integrable $P_r$-martingale $\gamma$-a.s. so that Jensen inequality ensures that $Q_{\nu}(\C \times D)=Q_{\nu, r}(\C)=1$. Hence, $Q_{\nu}$ and $Q_{\nu,r}$ are probability measures satisfying $dQ_{\nu}(x,r)=dQ_{\nu,r}(x)d\pi(r)$.

Fix $r \in D$, and define the probability measure $\bar{Q}_{\nu,r} \in \M_1^+(\C)$ by:
\begin{align*}
\der{\bar{Q}_{\nu,r}}{P_r}(x) := \exp\bigg\{ \int_0^T m_{\nu}(t,r) dW_t(x,r) - \frac{1}{2} \int_0^T m_{\nu}(t,r)^2 dt\bigg\}
\end{align*}
for which Novikov's criterion holds by boundedness of $m_{\nu}$. By Girsanov's theorem, $V^{\nu}(\cdot,r)$ (defined in \eqref{LnuVnuDef}) is a $\bar{Q}_{\nu,r}$-Brownian motion, and we can use Novikov's criterion again to check that:
\begin{align*}
\der{P_r}{\bar{Q}_{\nu,r}}(x) := \exp\bigg\{ -\int_0^T m_{\nu}(t,r) dV^{\nu}_t(x,r) - \frac{1}{2} \int_0^T m_{\nu}(t,r)^2 dt\bigg\},
\end{align*}
implying $\bar{Q}_{\nu,r} \simeq P_r$. Moreover, let $\bar{Q}_{\nu} \in \M_1^+(\C \times D)$ be such that $d\bar{Q}_{\nu}(x,r) = d\bar{Q}_{\nu,r}(x) d\pi(r)$. Then $\bar{Q}_{\nu} \simeq P$, and by the previous lemma $Q_{\nu} \ll \bar{Q}_{\nu}$ with density:
\begin{align*}
&\der{Q_{\nu}}{\bar{Q}_{\nu}}(x,r) = \E_{\gamma} \Bigg[ \exp\bigg\{\int_0^T G^{\nu}_t(r) dV^{\nu}_t(x,r)-\frac{1}{2}\int_0^T G^{\nu}_t(r)^2 dt\bigg\} \Bigg]\\
& \overset{\eqref{eq:QnurDensity}}{=} \exp\bigg\{\int_0^T \E_{\gamma} \Big[ \Lambda_t\big( G^{\nu}(r) \big) G^{\nu}_t(r) L^{\nu}_t(x,r) \Big]dV^{\nu}_t(x,r)-\frac{1}{2}\int_0^T \E_{\gamma} \Big[ \Lambda_t\big( G^{\nu}(r) \big) G^{\nu}_t(r) L^{\nu}_t(x,r) \Big]^2 dt\bigg\}\\
& \overset{\eqref{eq:DecompositionGamma}}{=} \exp\bigg\{\frac{1}{2} \int_{\hat{\Omega}} L^{\nu}_T(x,r)^2 d\gamma_{\widetilde{K}_{\nu,r}^T}  -\frac{1}{2} \int_0^T \widetilde{K}_{\nu,r}^t(t,t) dt \bigg\} \overset{\eqref{ineq:BoundOnLambda}}{\geq} \exp\big\{ -C_T\big\}>0.
\end{align*}
We will first prove that $I(Q_{\nu,r}|\bar{Q}_{\nu,r})$ is finite. This will bring, by applying the exact same reasoning as in \cite[Appendix B]{ben-arous-guionnet:95}, the equality:
\begin{equation*}
\forall \mu \in \M_1^+(\C \times D), \quad \quad \bar{H}_{\nu}(\mu) = I(\mu|Q_{\nu}),
\end{equation*} 
where
\begin{equation*}
\bar{H}_{\nu} : \mu \to\\
\left\{
\begin{array}{ll}
I(\mu|\bar{Q}_{\nu}) - \int_{\C \times D} \log\Big( \der{Q_{\nu}}{\bar{Q}_{\nu}}(x,r) \Big) d\mu(x,r)  & \text{if } I(\mu|\bar{Q}_{\nu}) < + \infty,\\
+ \infty & \text{otherwise. }
\end{array}
\right.
\end{equation*}
We will then prove that for every $\mu \in \M_1^+(\C \times D)$,
\begin{equation}\label{eq:QnuEq}
\bar{H}_{\nu}(\mu)=H_{\nu}(\mu),
\end{equation}
which will conclude the proof.\\
For the first point, observe that Girsanov's theorem ensures that the process $\Big(B^{\nu}_t(\cdot,r):=V^{\nu}_t(\cdot,r)- \int_0^t \E_{\gamma} \Big[ \Lambda_s\big( G^{\nu}(r) \big) G^{\nu}_s(r) L^{\nu}_s(\cdot,r) \Big]ds \Big)_{0 \leq t \leq T}$ is a $Q_{\nu,r}$-Brownian motion, so that
\begin{align*}
& I(Q_{\nu,r}|\bar{Q}_{\nu,r}) \\
& = \int_{\C} \bigg\{\int_0^T \E_{\gamma} \Big[ \Lambda_t\big( G^{\nu}(r) \big) G^{\nu}_t(r) L^{\nu}_t(x,r) \Big]dV^{\nu}_t(x,r)-\frac{1}{2}\int_0^T \E_{\gamma} \Big[ \Lambda_t\big( G^{\nu}(r) \big) G^{\nu}_t(r) L^{\nu}_t(x,r) \Big]^2 dt\bigg\} dQ_{\nu,r}(x) \\
& = \int_{\C} \bigg\{\int_0^T \E_{\gamma} \Big[ \Lambda_t\big( G^{\nu}(r) \big) G^{\nu}_t(r) L^{\nu}_t(x,r) \Big]dB^{\nu}_t(x,r)+\frac{1}{2}\int_0^T \E_{\gamma} \Big[ \Lambda_t\big( G^{\nu}(r) \big) G^{\nu}_t(r) L^{\nu}_t(x,r) \Big]^2dt\bigg\} dQ_{\nu,r}(x) \\
& = \frac{1}{2} \int_0^T \underbrace{\int_{\C} \E_{\gamma} \bigg[\Lambda_t\big(G^{\nu}(r)\big) G^{\nu}_t(r) \bigg(\int_0^t G^{\nu}_s(r) dV^{\nu}_s(x,r) \bigg) \bigg]^2 dQ_{\nu,r}(x) }_{\varphi_{\nu}(t,r)} dt.
\end{align*}

We now intend to bound $\varphi_{\nu}(t,r)$ uniformly in order to obtain the result.
\begin{align*}
& \varphi(t,r) \leq \int_{\C}  \Bigg\{ \E_{\gamma} \bigg[\Lambda_t\big(G^{\nu}(r)\big) G^{\nu}_t(r) \bigg(\int_0^t G^{\nu}_s(r) dB^{\nu}_s(x,r)\bigg) \bigg]^2\\
& +  \bigg(\int_0^t \widetilde{K}^t_{\nu,r}(t,s) \E_{\gamma} \Big[ \Lambda_s\big( G^{\nu}(r) \big) G^{\nu}_s(r) L^{\nu}_s(x,r) \Big] ds  \bigg)^2 \Bigg\} dQ_{\nu,r}(x)\\
& \overset{\mathrm{C.S.}, \eqref{ineq:BoundOnLambda}}{\leq}  C_T \int_{\C} \Bigg\{ \widetilde{K}^t_{\nu,r}(t,t) \E_{\gamma} \bigg[\bigg(\int_0^t G^{\nu}_s(r) dB^{\nu}_s(x,r)\bigg)^2 \bigg]\\
& + \int_0^t \widetilde{K}^t_{\nu,r}(t,s)^2 \E_{\gamma} \Big[ \Lambda_s\big( G^{\nu}(r) \big) G^{\nu}_s(r) L^{\nu}_s(x,r) \Big]^2ds \Bigg\}dQ_{\nu,r}(x) \\
& \overset{\text{Fubini}, \eqref{ineq:BoundOnLambda}}{\leq} C_T \Bigg\{\E_{\gamma} \bigg[ \int_{\C} \int_0^t G^{\nu}_s(r)^2 ds dQ_{\nu,r}(x) \bigg] + \int_0^t \varphi(s,r) ds\Bigg\} \overset{\text{Fubini}, \eqref{ineq:BoundOnLambda}}{\leq} C_T \bigg\{1 + \int_0^t \varphi(s,r) ds\bigg\},
\end{align*}
where we have used It\^o isometry, and where $C_T$ is uniform in space. 
Relying on Gronwall's lemma, we find that $\varphi(t,r)$ is uniformly bounded in space:
\begin{equation*}
\sup_{0 \leq t \leq T} \varphi_{\nu}(t,r) \leq  C_T\exp{C_T}.
\end{equation*}
This implies that exists a finite constant $\tilde{C}_T$, uniform in space, such that $I(Q_{\nu,r}|\bar{Q}_{\nu,r}) \leq \tilde{C}_T$.

Moreover, $I(Q_{\nu,r}|\bar{Q}_{\nu,r})$ is positive and $\pi-$measurable. We can thus integrate on $D$ to find:
\begin{equation*}
I(Q_{\nu}|\bar{Q}_{\nu}) = \int_{D} I(Q_{\nu,r}|\bar{Q}_{\nu,r}) d\pi(r) \leq \tilde{C}_T < \infty.
\end{equation*}
Remark that the proof of Proposition~\ref{prop:GammaBehaviour} readily applies to show that exists $0<\iota<1$ and $e>0$ such that
\[
\int_{\C \times D} \log\Big( \der{Q_{\nu}}{\bar{Q}_{\nu}}(x,r) \Big) d\mu(x,r) \leq \iota I(\mu|\bar{Q}_{\nu}) +e.
\]
In particular, $\bar{H}_{\nu}$ is finite whenever $I(\cdot|\bar{Q}_{\nu})$ is. Moreover, we can directly apply \cite[Appendix B]{ben-arous-guionnet:95}, to obtain:
\begin{equation*}
\forall \mu \in \M_1^+(\C \times D),  \; \; \bar{H}_{\nu}(\mu) = I(\mu|Q_{\nu}).
\end{equation*} 
We now show that equation \eqref{eq:QnuEq} holds. If both $I(\mu|P)$ and $I(\mu|\bar{Q}_{\nu})$ are infinite the results is clear. We handle the other cases by remarking that $\{ I(\cdot|P) < +\infty \} = \{ I(\cdot|\bar{Q}_{\nu}) < +\infty \}$. In fact, let $\mu \ll P \simeq \bar{Q}_{\mu}$, suppose that $I(\mu|\bar{Q}_{\nu})$ is finite, and observe that
\begin{align*}
I(\mu|P) & = \int_{\C \times D} \bigg\{ \log \Big( \der{\mu}{\bar{Q}_{\nu}}(x,r) \Big)+\log\Big( \der{\bar{Q}_{\nu}}{P}(x,r) \Big) \bigg\} d\mu(x,r).
\end{align*}
On the one hand, $\log \Big( \der{\mu}{\bar{Q}_{\nu}} \Big)$ is $\mu$-integrable as
\begin{align*}
\int_{\C \times D} \Big|\log \Big( \der{\mu}{\bar{Q}_{\nu}}(x,r) \Big) \Big| d\mu(x,r)= \int_{\C \times D} \Big|\log \Big( \der{\mu}{\bar{Q}_{\nu}}(x,r) \Big) \der{\mu}{\bar{Q}_{\nu}}(x,r) \Big| d\bar{Q}_{\nu}(x,r),
\end{align*}
as $\forall x \in \R_+$, $x\log(x) \geq -\frac{1}{\exp(1)}$, and as $I(\mu|\bar{Q}_{\nu})< + \infty$. Let us show that $\log\Big( \der{\bar{Q}_{\nu}}{P} \Big)$ is also $\mu$-integrable. In fact,
\begin{align*}
\int_{\C \times D} & \log \Big(\der{\bar{Q}_{\nu}}{P}(x,r) \Big)^- d\mu(x,r) = \int_{\C \times D} \log \Big( 1 \vee \der{P}{\bar{Q}_{\nu}}(x,r) \Big) d\mu(x,r) \\
& \overset{\eqref{eq:IneqRelativeEntropy}}{\leq} I(\mu|\bar{Q}_{\nu}) + \log \bigg(1+ \int_{\C \times D} \der{P}{\bar{Q}_{\nu}}(x,r) d\bar{Q}_{\nu}(x,r)\bigg) \leq  I(\mu|\bar{Q}_{\nu}) + \log(2)<+\infty,
\end{align*}
whereas
\begin{align*}
& \int_{\C \times D} \log \Big(\der{\bar{Q}_{\nu}}{P}(x,r) \Big)^+ d\mu(x,r) \overset{\eqref{eq:IneqRelativeEntropy}}{\leq} I(\mu|\bar{Q}_{\nu}) + \log \bigg(1+ \int_{\C \times D} \der{\bar{Q}_{\nu}}{P}(x,r) d\bar{Q}_{\nu}\bigg)\\
& \leq I(\mu|\bar{Q}_{\nu}) + \log \bigg(1+ \int_{D} \int_{\C} \exp\bigg\{ \int_0^T m_{\nu}(t,r) dV^{\nu}_t(x,r) - \frac{1}{2} \int_0^T m_{\nu}(t,r)^2 dt\bigg\} d\bar{Q}_{\nu,r}(x) e^{ \int_0^T m_{\nu}(t,r)^2 dt} d\pi(r)\bigg)\\
& \leq I(\mu|\bar{Q}_{\nu}) + \log \bigg(1+ \exp\Big\{ \frac{T \Jmax^2}{2 \ls^2} \Big\}\bigg)<+\infty.
\end{align*}
Hence, $I(\mu|\bar{Q}_{\nu})< + \infty$ implies that
\[
I(\mu|P) = I(\mu|\bar{Q}_{\nu}) + \int_{\C \times D} \log\Big( \der{\bar{Q}_{\nu}}{P}(x,r) \Big) d\mu(x,r) < \infty,
\]
and by symmetry, $I(\mu|P)<\infty$ implies finiteness of $I(\mu|\bar{Q}_{\nu})$ with same equality. Moreover, we can apply a similar reasoning as in the proof of Proposition~\ref{prop:GammaBehaviour} to show that exists constants $0<\iota<1$ and $e>0$ such that
\[
\Big|\int_{\C \times D} \log\Big( \der{\bar{Q}_{\nu}}{P}(x,r) \Big) d\mu(x,r) \Big| \leq \iota I(\mu|\bar{Q}_{\nu})+e,
\]
and
\[
|\Gamma_{\nu}(\mu)| \leq \iota I(\mu|P)+e.
\]
Hence, for $\mu \in \{ I(\cdot|P)<+\infty\}=\{I(\cdot|\bar{Q}_{\nu})<+\infty\}$ these quantities are finite. Moreover:
\[
\Gamma_{\nu}(\mu) =  \int_{\C \times D} \bigg\{ \log \Big( \der{Q_{\nu}}{\bar{Q}_{\nu}}(x,r) \Big)+ \log \Big(\der{\bar{Q}_{\nu}}{P}(x,r) \Big) \bigg\} d\mu(x,r)
\]
and we can split this integral as $\der{Q_{\nu}}{\bar{Q}_{\nu}}$ is bounded away from $0$ and $\log \Big(\der{\bar{Q}_{\nu}}{P} \Big)$ is $\mu$-integrable to obtain:
\begin{align*}
\Gamma_{\nu}(\mu) = \int_{\C \times D} \log \Big( \der{Q_{\nu}}{\bar{Q}_{\nu}}(x,r) \Big) d\mu(x,r) + \int_{\C \times D} \log \Big(\der{\bar{Q}_{\nu}}{P}(x,r) \Big) d\mu(x,r).
\end{align*}
As a consequence, for any $\mu \in \{ I(\cdot|P)<+\infty\}$, we have
\begin{align*}
I(\mu|P) -\Gamma_{\nu}(\mu) =I(\mu|\bar{Q}_{\nu}) - \int_{\C \times D} \log \Big( \der{Q_{\nu}}{\bar{Q}_{\nu}}(x,r) \Big) d\mu(x,r),
\end{align*}
which concludes the proof.
\end{proof}

We have thus proved that $H_{\nu}$ is a good rate function, and would like to extend this property to $H$: $H_{\nu}$ is seen in our proof as an intermediate tool, equal to $H$ when $I(\mu|P)=\infty$, but differing of $\Gamma-\Gamma_{\nu}$ otherwise. We control this difference below in Lemma~\ref{lemma1}.

Let us introduce two preliminary objects that will appear in the obtained upper-bound. 
First, because of spatial extension, it is useful to introduce a proper distance on $\C \times D$:
\begin{definition}
The map
\begin{equation}\label{def:SkorokhodLikeNorm}
d_T:
\left\{
	\begin{array}{ll}
	(\C \times D)^2 & \to \R_+ \\
	\big((x,r),(y,r')\big) & \to \bigg\{ \Vert r-r' \Vert_{\R^d}^2 + \underset{\underset{|b-a|\leq K_{\tau} \Vert r-r' \Vert_{\R^d}}{a,b \in [-\bar{\tau},0], t\in [0,T] }}{\sup}  \big| x_{t+a}-y_{t+b} \big|^2 \bigg\}^{\frac{1}{2}},
	\end{array}
\right.
\end{equation}
is a distance on $\C \times D$. Moreover,
\[
d_T\big((x^n,r_n),(x,r)\big) \to 0 \iff  \Vert x-x^n \Vert_{\infty,T} + \Vert r-r_n \Vert_{\R^d}\to 0,
\]
and $\big(\C \times D,d_T\big)$ is complete.
\end{definition}

\begin{remark}
In particular, $d_T$ generates the natural Borel $\sigma$-field of $\C \times D$, and $\big( \C \times D, d_T \big)$ is a Polish space.
\end{remark}

\begin{proof}
Symmetry and separation are easy to obtain. The triangular inequality is a consequence of the two following facts. First, for any $(x,r),(y,r'),(z,\tilde{r})$, we have
\begin{align*}
\underset{|b-a|\leq K_{\tau}\Vert r-r' \Vert_{\R^d}}{\underset{a,b \in [-\bar{\tau},0], t\in [0,T] }{\sup}}  \big| x_{t+a}-y_{t+b} \big| &\leq \underset{|c-a|\leq K_{\tau}\Vert r-\tilde{r} \Vert_{\R^d}, |b-c|\leq K_{\tau}\Vert \tilde{r}-r' \Vert_{\R^d}}{\underset{a,b,c \in [-\bar{\tau},0], t\in [0,T] }{\sup}}  \big| x_{t+a}-y_{t+b} \big| \\
& \leq \underset{|c-a|\leq K_{\tau}\Vert r-\tilde{r} \Vert_{\R^d}}{\underset{a,c \in [-\bar{\tau},0], t\in [0,T] }{\sup}}  \big| x_{t+a}-z_{t+c} \big| + \underset{|b-c|\leq K_{\tau}\Vert \tilde{r}-r' \Vert_{\R^d}}{\underset{c,b \in [-\bar{\tau},0], t\in [0,T] }{\sup}}  \big| z_{t+c}-y_{t+b} \big|,
\end{align*}
as the set
\[
\Big\{ a,b \in [-\bar{\tau},0], |b-a|\leq K_{\tau}\Vert r-r' \Vert_{\R^d} \Big\}
\]
is contained in
\[
 \Big\{ a,b \in [-\bar{\tau},0], \exists c \in [-\bar{\tau},0], |c-a|\leq K_{\tau}\Vert r-\tilde{r} \Vert_{\R^d}, |b-c|\leq K_{\tau}\Vert \tilde{r}-r' \Vert_{\R^d} \Big\}.
\]
Second, the triangular inequality of $\R^2$ for the Euclidean norm gives $\forall a_1,b_1,a_2,b_2 \in \R$,
\[
\Big\{(a_1+b_1)^2+(a_2+b_2)^2 \Big\}^{\frac{1}{2}} \leq \big\{a_1^2+a_2^2 \big\}^{\frac{1}{2}}  + \big\{b_1^2+b_2^2 \big\}^{\frac{1}{2}}.
\]
Hence, $d_T$ is a distance on $\C \times D$. Let $(x^n,r_n)_{n \in \N} \in \big(\C \times D\big)^{\N}$, and $(x,r) \in \C \times D$. Taking $a=b$ in the supremum, we see that
\begin{equation}\label{eq:DistanceInequality1}
\Vert r-r' \Vert_{\R^d} + \Vert x-y \Vert_{\infty,T} \leq d_T\big((x,r),(y,r')\big),
\end{equation}
we have that $d_T\big((x^n,r_n),(x,r)\big) \to 0$ implies $r_n \to r$, and $x_n \to x$ for the supremum-norm on $[-\bar{\tau},T]$. Conversely, suppose that $r_n \to r$, and $x_n \to_{\lVert \cdot \rVert_{\infty,T}} x$, and let $\eta>0$ such that
\[
s,t \in [-\bar{\tau},T], |s-t| \leq \eta \quad \implies \quad |x_t-x_s| \leq \varepsilon.
\]
We then see that, we can find a $n_0$ such that $\forall n \geq n_0$,
\[
d_T\big((x^n,r_n),(x,r)\big)^2 \leq \Vert r-r_n \Vert_{\R^d}^2 + 2 \Vert  x-x^n \Vert_{\infty,T}^2 + 2 \underset{|b-a|\leq \eta}{\underset{a,b \in [-\bar{\tau},0], t\in [0,T] }{\sup}}  \big| x_{t+a}-x_{t+b} \big|^2 \leq 3 \varepsilon^2.
\]
The completion of $\big( \C \times D,d_T \big)$ comes from that of $\big(\C ,\lVert \cdot \rVert_{\infty,T}\big)$, and $\big(D, |\cdot |\big)$, and from \eqref{eq:DistanceInequality1}.
\end{proof}

We also define the $2$-Vaserstein distance on $\M_1^+(\C \times D)$, associated with $d_T$:
\begin{equation*}
d_T^V(\mu,\nu):=\inf_{\xi}\bigg\{ \int_{(\C \times D)^2} d_T\big((x,r),(y,r') \big)^2 d\xi\big((x,r),(y,r')\big) \bigg\}^{\frac{1}{2}}
\end{equation*}
the infimum being taken on the laws $\xi \in \M_1^+\big((\C \times D)^2\big)$ with marginals $\mu$ and $\nu$. In the following, we will, for any $t \in [0,T]$, denote by $d_t$ and $d_t^V$ the respective restrictions of $d_T$ and $d_T^V$ on $\big(\C\big([-\bar{\tau},t],\R\big) \times D\big)^2$.

Here are a few regularity properties of the covariance and mean of our Gaussian interactions:

\begin{proposition}\label{prop:MeanAndVarRegularity}
There exists $C_T>0$ such that for any $\mu, \nu \in \M_1^+(\C \times D)$, $r \in D$, $t \in [0,T]$ and $u,s \in [0,t]$:
\begin{multline}\label{ineq:MajDiff}
   \big|m_{\mu}(t,r)-m_{\nu}(t,r)\big| +\big|K_{\mu}(t,s,r)-K_{\nu}(t,s,r)\big|
   + \big|\widetilde{K}^{t}_{\mu,r}(s,u)-\widetilde{K}^{t}_{\nu,r}(s,u)\big| \leq C_T d_T^V(\mu,\nu).
\end{multline}
\end{proposition}

\begin{proof}
Let $\xi \in \M_1^+\big((\C \times D)^2\big)$ with marginals $\mu$ and $\nu$, and let $\big(G, G'\big)$ be, under $\gamma$, a family of independent bi-dimensional centered Gaussian processes with covariance $K_{\xi}(s,t,r)$ given by:
\begin{equation}
\frac{1}{\lambda(r)^2} \int_{(\C \times D)^2} 
\left(
\begin{array}{ccc}
\sigma_{r r'}^2 S(x_{s-\tau_{r r'}})S(x_{t-\tau_{r r'}})  & \sigma_{r r'}\sigma_{r \tilde{r}'}S(x_{s-\tau_{r r'}})S(y_{t-\tau_{r \tilde{r}'}}) \\
\sigma_{r r'}\sigma_{r \tilde{r}'} S(y_{s-\tau_{r \tilde{r}'}})S(x_{t-\tau_{r r'}}) & \sigma_{r \tilde{r}'}^2 S(y_{s-\tau_{r \tilde{r}'}})S(y_{t-\tau_{r \tilde{r}'}}) \\
\end{array}
\right) 
d\xi \big((x,r'),(y,\tilde{r}')\big) \label{kxi}.
\end{equation}
with the short-hand notations $\sigma_{r r'}:=\sigma(r,r')$, $\tau_{r r'}=\tau(r,r')$.
Let us first take care of the mean difference:
    \begin{align*}
	& \big|m_{\mu}(t,r)-m_{\nu}(t,r)\big| = \bigg|\frac{1}{\lambda(r)} \int_{\C \times D} J(r,r') S (x_{t-{\tau(r,r')}}) d(\mu-\nu)(x,r')\bigg|  \\
	& \leq \frac{1}{\ls}  \int_{(\C \times D)^2} \Big|J(r,r')S (x_{t-{\tau(r,r')}}) - J(r,\tilde{r}')S (y_{t-{\tau(r,\tilde{r}')}})\Big| d\xi\big((x,r'),(y,\tilde{r}')\big) \\
	& \leq \frac{1}{\ls} \int_{(\C \times D)^2} \Big\{K_J \Vert r'-\tilde{r}' \Vert_{\R^d}+ \Jmax \big|S\big(x_{t-{\tau(r,r')}}\big) - S\big(y_{t-{\tau(r,\tilde{r}')}}\big)\big| \Big\} d\xi\big((x,r'),(y,\tilde{r}')\big) \\
	& \leq \frac{1}{\ls} \int_{(\C \times D)^2} \bigg\{K_J \Vert r'-\tilde{r}' \Vert_{\R^d}+ \Jmax K_S \underset{|b-a|\leq K_{\tau} \Vert r'-\tilde{r}' \Vert_{\R^d}}{\underset{a,b \in [-\bar{\tau},0], t\in [0,T] }{\sup}} \big|x_{t-a} - y_{t-b}\big| \bigg\} d\xi\big((x,r'),(y,\tilde{r}')\big) \\
	& \overset{\mathrm{C.S.}}{\leq} C_T \bigg\{\int_{(\C \times D)^2} d_T\big((x,r'),(y,\tilde{r}') \big)^2 d\xi\big((x,r'),(y,\tilde{r}')\big)\bigg\}^{\frac{1}{2}}.
	\end{align*}

Moreover,
\begin{align*}
\big|K_{\mu}(t,s,r) - K_{\nu}(t,s,r) \big| & = \Big| \E_{\gamma} \Big[ G_sG_t-G_s' G_t' \Big]\Big| \leq  C_T \bigg\{\E_{\gamma} \Big[ \big(G_t-G_t'\big)^2 \Big]^{\frac{1}{2}} +\E_{\gamma} \Big[ \big(G_s-G_s'\big)^2 \Big]^{\frac{1}{2}}\bigg\}.
\end{align*}
and
\begin{align*}
\big|\widetilde{K}^t_{\mu,r}(s,u) - \widetilde{K}^t_{\nu,r}(s,u) \big| & \overset{\eqref{ineq:BoundOnLambda}}{\leq} C_T \bigg\{ \E_{\gamma} \Big[ \big( \Lambda_t(G)-\Lambda_t(G')\big)^2 \Big]^{\frac{1}{2}} + \E_{\gamma} \Big[ \big(G_s-G_s'\big)^2 \Big]^{\frac{1}{2}} +\E_{\gamma} \Big[ \big(G_u-G_u'\big)^2 \Big]^{\frac{1}{2}} \bigg\}\\
& \overset{\eqref{ineq:DiffLambda}}{\leq} C_T \bigg\{ \bigg(\int_0^t \E_{\gamma} \Big[ \big(G_v-G'_v\big)^2 \Big] dv \bigg)^{\frac{1}{2}} + \E_{\gamma} \Big[ \big(G_s-G_s'\big)^2 \Big]^{\frac{1}{2}} +\E_{\gamma} \Big[ \big(G_u-G_u'\big)^2 \Big]^{\frac{1}{2}} \bigg\}.
\end{align*}

\begin{align*}
\E_{\gamma} \Big[ \big(G_t-G_t'\big)^2 \Big] = \frac{1}{\lambda(r)^2}  \int_{(\C \times D)^2 } \Big(\sigma(r,r')S(x_{t-\tau(r,r')})-\sigma(r,\tilde{r}')S(y_{t-\tau(r,\tilde{r}')})\Big)^2 \, d\xi\big((x,r'),(y,\tilde{r}')\big).
\end{align*}
  Splitting the integrand of the right-hand side, we find:
  \begin{multline*}
  \Big(\sigma_{r r'} S(x_{t-\tau_{r r'}})-\sigma_{r \tilde{r}'}S(y_{t-\tau_{r \tilde{r}'}} )\Big)^2 \leq 2 \Big\{ \big(\sigma_{r r'}-\sigma_{r \tilde{r}'}\big)^2 S(x_{t-\tau_{r r'}})^2+\sigma_{r \tilde{r}'}^2 \big(S(x_{t-\tau_{r r'}})-S(y_{t-\tau_{r \tilde{r}'}})\big)^2 \Big\} \\
	\leq 2 K^2_{\sigma} \Vert r'-\tilde{r}' \Vert_{\R^d}^2 + 2\smax^2 K_S^2 \underset{|b-a|\leq K_{\tau}\Vert r'-\tilde{r}' \Vert_{\R^d}}{\underset{a,b \in [-\bar{\tau},0]}{\sup}} |x_{t+a}-y_{t+b}|^2 \leq C d_t\big((x,r'),(y,\tilde{r}') \big)^2,
  \end{multline*}
so that
\begin{align*}
\E_{\gamma} \Big[ \big(G_t-G_t'\big)^2 \Big] \leq C_T \int_{(\C \times D)^2 }  d_t\big((x,r'),(y,\tilde{r}') \big)^2 d\xi\big((x,r'),(y,\tilde{r}')\big).
\end{align*}

Taking the infimum on $\xi$ yields \eqref{ineq:MajDiff}.
	
\end{proof}

We have now introduced all the needed elements to state the following technical lemma concluding on the fact that $H$ is a good rate function. 
\begin{lemma}\label{lemma1}
Let $\mu,\nu \in \mathcal{M}_1^+\big( \C \times D\big)$, then:
\begin{enumerate}
\item  there exists a positive constant $C_T$ such that:
\begin{enumerate}
\item  $|\Gamma_{1,\nu}(\mu)-\Gamma_1(\mu)| \leq C_T d_T^V(\mu,\nu)$.
\item  $|\Gamma_{2,\nu}(\mu)-\Gamma_2(\mu)| \leq C_T \big(1+I(\mu|P)\big) d_T^V(\mu,\nu)$.
\end{enumerate}
\item $H$ is a good rate function.
\end{enumerate}
\end{lemma}

\begin{proof}
The main techniques were introduced in~\cite[Lemma 3.3-3.4]{ben-arous-guionnet:95} and used in a neuroscience setting in~\cite[ Lemma.5]{cabana-touboul:12}, but spatiality induces new issues essentially impacting the proof of point~(1.b). For the sake of completeness, we will here reproduce these techniques, and address the specific spatial difficulties.

 \noindent {\bf Proof of Lemma~\ref{lemma1}.(1.a).} \\

    We define
    \[
     \Gamma_1(\mu,r):= - \frac{1}{2} \int_0^T \Big(\widetilde{K}^t_{\mu,r}(t,t)+ m_{\mu}(t,r)^2\Big) dt,
    \]
    so that
    \begin{align}
     & \big|\Gamma_{1,\nu}(\mu)- \Gamma_1(\mu) \big| = \bigg|\int_{\C \times D} \big(\Gamma_1(\nu,r)-\Gamma_1(\mu,r)\big) d\mu(x,r) \bigg| \nonumber \\
     & \leq \frac{1}{2}\int_0^T \Big|\big(m_{\mu}(t,r) -m_{\nu}(t,r)\big)\big(m_{\mu}(t,r) +m_{\nu}(t,r)\big)\Big| + \Big|\widetilde{K}^t_{\nu,r}(t,t)-\widetilde{K}^t_{\mu,r}(t,t) \Big| dt \overset{\eqref{ineq:MajDiff}}{\leq} C_T d_T^V(\mu,\nu).\label{gamma1}
    \end{align} 
	
\noindent {\bf Proof of Lemma~\ref{lemma1}.(1.b)}\\

Note that if $I(\mu \vert P)=\infty$ the inequality is obvious. Let then $\mu \in \M_1^+(\C \times D)$ with $I(\mu \vert P)<\infty$ implying $\mu \ll P$ and finiteness of $\Gamma_{\nu}(\mu)$, and $\Gamma(\mu)$. This also implies that $\mu$ has a measurable density $\rho_{\mu}$ with respect to $\mathcal{B}\big(\C \times D \big)$:
\[
d\mu(x,r)=\rho_{\mu}(x,r)dP(x,r)=\rho_{\mu}(x,r)dP_r(x) d\pi(r).
\]
Hence, for $r \in D$ such that $c_{\mu}(r):=\int_{\C} \rho_{\mu}(x,r) dP_r(x) \not\in \{0,+\infty\}$, we can properly define $\mu_r \in \M_1^+(\C)$ by $d\mu_r(x):=\frac{\rho_{\mu}(x,r)}{c_{\mu}(r)}dP_r(x)$. Of course $\mu_r \ll P_r$, and
\begin{equation}\label{eq:DecompositionMu}
d\mu(x,r)=d\mu_r(x) c_{\mu}(r) d\pi(r).
\end{equation}
Remark that $c_{\mu}$ is a measurable function of space such that $\int_D c_{\mu}(r) d\pi(r)=1$, and that the set $\{ r \in D, \; c_{\mu}(r) \in \{0,+\infty\} \}$ do not impact the value of integrals over $\mu$.

In order to obtain the proper inequality, we split the difference of interest into different terms:
\begin{multline*}
\big|\Gamma_{2,\nu}(\mu) - \Gamma_2(\mu)\big| \leq \frac{1}{2} \bigg|\int_{\C \times D} \bigg\{ \int_{\hat{\Omega}} L^{\nu}_T(x,r)^2 d\gamma_{\widetilde{K}_{\nu,r}^{T}} - \int_{\hat{\Omega}} L^{\mu}_T(x,r)^2  d\gamma_{\widetilde{K}_{\mu,r}^{T}} \bigg\}d\mu(x,r) \bigg|\\
+ \bigg|\int_{\C \times D} \int_0^T (m_{\nu}-m_{\mu})(t,r)dW_t(x,r) d\mu(x,r)\bigg|
\leq \! \frac{1}{2} \bigg|\int_{\C \times D} \int_{\hat{\Omega}} L^{\nu}_T(x,r)^2 d\big(\gamma_{\widetilde{K}_{\nu,r}^{T}}-\gamma_{\widetilde{K}_{\mu,r}^{T}}\big) d\mu(x,r) \bigg|  \\
+\frac{1}{2} \bigg|\int_{\C \times D} \int_{\hat{\Omega}} \Big\{L^{\mu}_T(x,r)^2 - L^{\nu}_T(x,r)^2 \Big\} d\gamma_{\widetilde{K}_{\mu,r}^{T}} d\mu(x,r) \bigg| + \bigg|\int_{\C \times D} \int_0^T (m_{\nu}-m_{\mu})(t,r)dW_t(x,r) d\mu(x,r)\bigg|.
\end{multline*}

Let $\xi \in \M_1^+\big((\C \times D)^2\big)$ with marginals $\mu$ and $\nu$ be such that
\begin{align*}
\int_{(\C \times D)^2 }  d_T\big((x,r'),(y,\tilde{r}') \big)^2 d\xi\big((x,r'),(y,\tilde{r}')\big) \leq \big(d_T^V(\mu,\nu)+\varepsilon\big)^2.
\end{align*}
Moreover, let $\big(G(r), G'(r)\big)_{r \in D}$ be, under $\gamma$, a family of independent bi-dimensional centered Gaussian processes with covariance $K_{\xi}(s,t,r)$ as define in \eqref{kxi}. Remark that for any map $f$
\[
\E_{\gamma} \Big[ f\big(G^{\mu}(r)\big)\Big] - \E_{\gamma} \Big[ f\big(G^{\nu}(r)\big)\Big]= \E_{\gamma} \Big[f\big(G(r)\big)-f\big(G'(r)\big) \Big]
\]
and as proved in Proposition~\ref{prop:MeanAndVarRegularity},
\[
\E_{\gamma} \Big[ \big(G_t(r)-G'_t(r) \big)^2 \Big] \leq C_T \big(d_T^V(\mu,\nu)+\varepsilon\big)^2.
\]
Let also
\begin{equation*}
L_t(x,r):= \int_0^t G_s(r) dV^{\mu}_s(x,r), \; \; L'_t(x,r):= \int_0^t G'_s(r) dV^{\nu}_s(x,r).
\end{equation*}
Using inequality \eqref{ineq:BoundOnLambda}, we then obtain:
\begin{multline}
|\Gamma_{2,\nu}(\mu) - \Gamma_2(\mu)| \overset{\mathrm{C.S.}}{\leq} C_T \Bigg\{ \overbrace{  \int_{\C \times D} \E_{\gamma} \Big[ \big|\Lambda_T(G(r))-\Lambda_T(G'(r))\big| L'_T(x,r)^2 \Big] d\mu(x,r) }^{B_1:=}\\ + \underbrace{\prod_{\varepsilon=\pm1} \Bigg( \int_{\C \times D} \E_{\gamma} \Bigg[ \bigg( \int_0^T (G_t(r) +\varepsilon G'_t) dV^{\mu}_t(x,r)  \bigg)^2 \Bigg] d\mu(x,r) \Bigg)^\frac{1}{2}}_{=:B_2} \\
+ \underbrace{ \Bigg| \int_{\C \times D}  \E_{\gamma} \bigg[ \bigg( \int_0^T G_t(r) dV^{\mu}_t(x,r) \bigg)^2 \!\!\!-\! \bigg( \int_0^T G_t(r) dV^{\nu}_t(x,r) \bigg)^2 \bigg] d\mu(x,r) \Bigg|}_{=:B_3} \\+ \underbrace{ \bigg( \int_{\C \times D} \Big| \int_0^T (m_{\nu}-m_{\mu})(t,r)dW_t(x,r) \Big|^2 d\mu(x,r) \bigg)^{\frac{1}{2}}}_{=:B_4}  \Bigg\} \label{ineqgamma2}.
\end{multline}

Before bounding these four terms, we prove a useful inequality. For any $h, m \in L^2([0;T], dt)$, with $m$ bounded, and any $r \in D$ with $c_{\mu}(r) \not\in \{0,+\infty\}$,
\begin{equation}\label{eq:RelatEntrIneq}
\int_{\C} \Big( \int_0^T h_t (dW_t(x,r) - m(t)dt) \Big)^2 d\mu_r(x) \leq  2\bigg\{ \int_{\C} \Big(\int_0^T h_t dW_t(x,r) \Big)^2 + \Big( \int_0^T h_t m_t dt \Big)^2 d\mu_r(x) \bigg\}.
\end{equation}

Moreover, supposing that $h \neq  0_{L^2([0;T], dt)}$, then $\Phi_h(x)=\frac{\Big(\int_0^T h_t dW_t(x,r) \Big)^2}{4 \big(\int_0^T h_t^2 dt \big)}$ is a well-defined, positive and measurable function of the $\sigma$-algebra $\mathcal{B}(\C)$, so that resorting to \eqref{eq:IneqRelativeEntropy} one obtains
\begin{equation*}
\int_{\C} \Phi_h(x) d\mu_r(x) \leq  I(\mu_r \vert P_r) +\log \int_{\C} \exp \Phi_h(x) dP_r(x).
\end{equation*}

As $W(.,r)$ is a Brownian motion under $P_r$, $\Phi_h \sim \mathcal{N}\big(0,\frac 1 4 \big)^2$, so that Gaussian calculus gives, for any $C>2$:
\begin{equation*}
\int_{\C} \Big(\int_0^T h_t dW_t(x,r)\Big)^2 d\mu_r(x) \leq  C \big(I(\mu_r \vert P_r) + 1\big)\Big( \int_0^T h_t^2 dt \Big)
\end{equation*}

Remark that this inequality obviously holds when $h=0_{L^2([0;T], dt)}$.
Applying this result in \eqref{eq:RelatEntrIneq} one eventually finds:
\begin{align}\label{ineqphi}
\int_{\C} \Big( \int_0^T h_t (dW_t(x,r) - m(t)dt) \Big)^2 d\mu_r(x) & \overset{\mathrm{C.S.}}{\leq} 2 \bigg(C\big(1+I(\mu_r|P_r)\big) + m^2_{\infty}T \bigg) \Big( \int_0^T h_t^2 dt \Big)\nonumber \\
& \leq C_T \big(1+I(\mu_r|P_r)\big) \Big( \int_0^T h_t^2 dt \Big).
\end{align}

With this result in mind, we now control the first term. Recall that, by \eqref{ineq:DiffLambda},
\begin{align*}
\big|\Lambda_T(G(r))-\Lambda_T(G'(r))\big| \leq C_T \bigg( \int_0^T \big| G_t(r)^2-{G'}_t(r)^2\big|dt + \int_0^T \E_{\gamma} \Big[ \big( G_t(r)-{G'}_t(r)\big)^2 \Big]^{\frac{1}{2}}  dt \bigg).
\end{align*}

Now, relying on the decomposition of $\mu$, we find
\begin{align*}
B_1 & \overset{\text{Fubini}}{=} \int_{D} \E_{\gamma} \Bigg[ \Big|\Lambda_T(G(r))-\Lambda_T(G'(r))\Big| \bigg\{ \int_{\C} L'_T(x,r)^2 d\mu_r(x) \bigg\} \Bigg] c_{\mu}(r) d\pi(r)\\
& \overset{\eqref{ineqphi}}{\leq} \int_{D} C_T \big(I(\mu_r|P_r)+1 \big)  \E_{\gamma} \bigg[\Big|\Lambda_T(G(r))-\Lambda_T(G'(r))\Big|  \bigg\{\int_0^T  {G_t'}(r)^2 dt \bigg\} \bigg] c_{\mu}(r) d\pi(r) \\
& \leq C_T \int_{D} \bigg( \int_0^T \int_0^T \E_{\gamma} \bigg[ \big| G_s(r)^2-{G'}_s(r)^2\big|  {G_t'}(r)^2  \bigg] \big(I(\mu_r|P_r)+1 \big) ds dt + d_T^V(\mu,\nu) +\varepsilon \bigg) c_{\mu}(r) d\pi(r) \\
& \leq C_T  \bigg( \int_{D} I(\mu_r|P_r) c_{\mu}(r) d\pi(r) +1 \bigg) \big(d_T^V(\mu,\nu)+\varepsilon\big).
\end{align*}

where the last inequality is a consequence of Cauchy-Schwarz's inequality, and Isserlis' theorem. Observe that:

\begin{align*}
\int_{D} I(\mu_r|P_r) c_{\mu}(r) d\pi(r) & = \int_D \int_{\C} \log\Big(\frac{\rho_{\mu}(x,r)}{c_{\mu}(r)}\Big) d\mu_r(x) c_{\mu}(r) d\pi(r)\\
& = \int_{\C \times D} \log(\rho_{\mu}(x,r)) d\mu(x,r) -\int_{D} \log(c_{\mu}(r)) c_{\mu}(r) d\pi(r) \overset{\text{Jensen}}{\leq} I(\mu|P).
\end{align*}
As a consequence,
\begin{align*}
B_1 \leq C_T \big(1+I(\mu|P) \big)\big(d_T^V(\mu,\nu)+\varepsilon\big).
\end{align*}

Similarly, there exists a constant $c_T$ such that
\begin{align*}
B_2 & \leq \prod_{\varepsilon=\pm1}  \Bigg( \int_D c_T \big(1+I(\mu_r|P_r)\big) \E_{\gamma} \bigg[ \int_0^T \big(G_t(r) +\varepsilon G'_t(r)\big)^2 dt \bigg] c_{\mu}(r) d\pi(r)\Bigg)^{\frac{1}{2}} \\ 
& \leq C_T \big(1+I(\mu|P)\big)^{\frac 1 2 }  \Bigg( \int_D \big(1+I(\mu_r|P_r)\big) \int_0^T \E_{\gamma} \Big[ \big(G_t(r) - G'_t(r)\big)^2 \Big] dt c_{\mu}(r)  d\pi(r)\Bigg)^{\frac{1}{2}}\\
& \leq C_T \big(1+I(\mu|P)\big) \big(d_T^V(\mu,\nu)+\varepsilon\big).
\end{align*}

To bound $B_3$, we first use Cauchy-Schwarz inequality:
\begin{align}\label{ineq:B3}
B_3 \leq \prod_{\varepsilon = \pm 1} \Bigg\{ \int_{\C \times D} \E_{\gamma}  \Bigg[ \bigg| \int_0^T G_t(r) \Big( (1+\varepsilon)dW_t(x,r) - (m_{\nu}(t,r) + \varepsilon m_{\mu}(t,r))dt \Big) \bigg|^2 \Bigg] d\mu(x,r)  \Bigg\}^{\frac{1}{2}}. 
\end{align}
Then, again by Cauchy-Schwarz's inequality, one observes that
\begin{align*}
\E_{\gamma} \bigg[ \Big| \int_0^T G_t(r) \big(m_{\mu}(t,r) -m_{\nu}(t,r)\big) dt \Big|^2 \bigg]  & \overset{\eqref{ineq:MajDiff}}{\leq} C_T d_T^V(\mu,\nu)^2.
\end{align*}
Moreover, \eqref{ineqphi} gives:
\begin{align*}
\int_{\C} \bigg\{ \int_0^T 2 G_t(r)\bigg(dW_t(x,r) - \frac{m_{\mu}(t,r)+m_{\nu}(t,r)}{2} dt\bigg) \bigg\}^2 d\mu_r(x) & \leq c_T \big( 1+ I(\mu_r|P_r) \big) \int_0^T G^2_t(r) dt.
\end{align*}
Using Jensen's inequality and injecting the last two inequalities in \eqref{ineq:B3} gives:
\begin{align*}
B_3 & \leq C_T \big(1+I(\mu|P)\big)^{\frac{1}{2}} d_T^V(\mu,\nu) \leq C_T \big(1+I(\mu|P)\big) d_T^V(\mu,\nu)
\end{align*}
as $I(.|P) \geq 0$.

As of the last term, we have
\begin{align*}
B_4 & \overset{\eqref{ineqphi}}{\leq} \Bigg( \int_{D} c_T \big(1+I(\mu_r|P_r)\big) \bigg\{\int_0^T \Big(m_{\mu}(t,r) - m_{\nu}(t,r)\Big)^2 dt \bigg\} c_{\mu}(r) d\pi(r) \Bigg)^{\frac{1}{2}}\\
& \overset{\eqref{ineq:MajDiff}}{\leq} C_T \big(1+I(\mu|P)\big)^{\frac{1}{2}} d_T^V(\mu,\nu) \leq C_T \big(1+I(\mu|P)\big) d_T^V(\mu,\nu).
\end{align*}
Hence, we conclude that there exists a constant $C_T$ satisfying
\begin{equation*}
|\Gamma_{2,\nu}(\mu) - \Gamma_2(\mu)| \leq  C_T \big(1+I(\mu|P)\big) \big(d_T^V(\mu,\nu)+\varepsilon\big).
\end{equation*}
Sending $\varepsilon$ to $0$ thus gives the result.

\noindent{\bf Proof of Lemma\ref{lemma1}.(3):}
We proceed exactly as in Lemma 5.(vi)~\cite{cabana-touboul:12}, remarking that $\big(\C \times D, d_T\big)$ is a Polish space.

\end{proof}

\medskip

\subsection{Upper-bound and Tightness}

We have proved that $H=I(.|P)-\Gamma$ is a good rate function, and we now want to show that it is indeed associated with a LDP. We demonstrate here a weak LDP relying on an upper-bound inequality for compact subsets, and tightness of the family $\Big(Q^N\big( \hat{\mu}_N \in \cdot \big)\Big)_N$. To prove the first point, we take advantage of the full LDP followed by $\hat{\mu}_N$ under $\big( Q_{\nu} \big)^{\otimes N}$, and have to control an error term. The second point will rely on the exponential tightness of $P^{\otimes N}$. These proofs are inspired from those of Guionnet in a non-spatial spin-glass model~\cite{guionnet:97}.

\begin{theorem}\label{lemma3}
For any compact subset $K$ of $\M_1^+(\C \times D)$,
\begin{equation*}
\limsup_{N \rightarrow \infty} \frac{1}{N} \log{ Q^N(\hat{\mu}_N \in K)} \leq - \inf_{K} H .
\end{equation*}
\end{theorem}

\begin{proof}
Let $\delta > 0$. as, $\big(\M_1^+(\C \times D),d_T^V \big)$ is a Polish space, we can find an integer $M$ and a family $(\nu_i)_{1\leq i \leq M}$ of $\M_1^+(\C \times D)$ such that
$$ K \subset \bigcup_{i=1}^M B(\nu_i,\delta),$$
where $B(\nu_i,\delta):= \big\{ \mu | d_T^V(\mu,\nu_i) < \delta \big\}$.
A classical result (see e.g.~\cite[lemma 1.2.15]{dembo-zeitouni:09}), ensures that
\begin{equation}\label{ineq:UpperBound1}
\limsup_{N \to \infty} \frac{1}{N} \log{ Q^N\big(\hat{\mu}_N \in K\big)} \leq \max_{1 \leq i \leq M} \limsup_{N \to \infty} \frac{1}{N} \log{Q^N\big(\hat{\mu}_N \in K\cap B(\nu_i,\delta)\big)  } .
\end{equation}
Lemma~\ref{lemma2} yields:
\begin{align*}
Q^N\big(\hat{\mu}_N \in K \cap B(\nu,\delta)\big) & = \int_{\hat{\mu}_N \in K \cap B(\nu,\delta)} \exp\Big\{ N \bar{\Gamma}(\hat{\mu}_N)\Big\} dP^{\otimes N}(\mathbf{x},\mathbf{r}) \\
= \int_{\hat{\mu}_N \in K \cap B(\nu,\delta)} & \exp\Big\{ N \big(\bar{\Gamma}(\hat{\mu}_N)- \bar{\Gamma}_{\nu}(\hat{\mu}_N)\big)\Big\}  \exp\Big\{ N \bar{\Gamma}_{\nu}(\hat{\mu}_N)\Big\} dP^{\otimes N}(\mathbf{x},\mathbf{r}).
\end{align*}
Recall definition \eqref{def:Xmu} and let $(\tilde{X}^{\mu}_i)_{1 \leq i \leq N}$ be a family of independent variables of $\big( \hat{\Omega}, \hat{\mathcal{F}}, \gamma\big)$ with same law as $(X^{\mu}_i)_{1 \leq i \leq N}$. We will denote by $\big(\tilde{G}^{\mu}(r_i)\big)_{1\leq i \leq N}$ the associated independent Gaussian processes.
Then, for any conjugate exponents $(p,q)$,
\begin{align*}
& Q^N\big(\hat{\mu}_N \in K \cap B(\nu,\delta)\big) = \int_{\hat{\mu}_N \in K \cap B(\nu,\delta)} \exp\Big\{ N \big(\bar{\Gamma}(\hat{\mu}_N)- \bar{\Gamma}_{\nu}(\hat{\mu}_N)\big)\Big\} dQ_{\nu}^{\otimes N}(\mathbf{x},\mathbf{r}) \\
& \leq Q_{\nu}^{\otimes N}\big(\hat{\mu}_N \in K \cap B(\nu,\delta)\big)^{\frac{1}{p}}  \bigg(\int_{\hat{\mu}_N \in K \cap B(\nu,\delta)} \exp\Big\{ q N\big( \bar{\Gamma}(\hat{\mu}_N)- \bar{\Gamma}_{\nu}(\hat{\mu}_N)\big)\Big\} dQ_{\nu}^{\otimes N}(\mathbf{x},\mathbf{r}) \bigg)^{\frac{1}{q}} \\
& \leq Q_{\nu}^{\otimes N}\big(\hat{\mu}_N \in K \cap B(\nu,\delta)\big)^{\frac{1}{p}} \Bigg(\int_{\hat{\mu}_N \in K \cap B(\nu,\delta)} \prod_{i=1}^N \mathcal{E}_{\gamma} \bigg( \frac{\exp\big\{ \tilde{X}^{\hat{\mu}_N}_i\big\}}{\mathcal{E}_{\gamma} \big( \exp\big\{ \tilde{X}^{\nu}_i\big\} \big)} \bigg)^q dQ_{\nu}^{\otimes N}(\mathbf{x},\mathbf{r}) \Bigg)^{\frac{1}{q}}
\end{align*}
\begin{align}
& \overset{\text{Jensen}}{\leq} Q_{\nu}^{\otimes N}\big(\hat{\mu}_N \in K \cap B(\nu,\delta)\big)^{\frac{1}{p}} \Bigg( \int_{\hat{\mu}_N \in K \cap B(\nu,\delta)} \E_{\gamma} \Bigg[\Bigg( \prod_{i=1}^N \exp\big\{ q \big( \tilde{X}^{\hat{\mu}_N}_i - \tilde{X}^{\nu}_i \big) \big\} \!\!\Bigg) \frac{\prod_{i=1}^N  \exp{\tilde{X}^{\nu}_i}}{\E_{\gamma} \Big[ \prod_{i=1}^N \exp{\tilde{X}^{\nu}_i} \Big]}  \Bigg]  dQ_{\nu}^{\otimes N} \Bigg)^{\frac{1}{q}} \nonumber \\
& \leq Q_{\nu}^{\otimes N}\big(\hat{\mu}_N \in K \cap B(\nu,\delta)\big)^{\frac{1}{p}}   \Bigg(\underbrace{\int_{\hat{\mu}_N \in K \cap B(\nu,\delta)} \E_{\gamma} \Bigg[ \prod_{i=1}^N \exp\Big\{ q \Big( \tilde{X}^{\hat{\mu}_N}_i - \tilde{X}^{\nu}_i \Big) +\tilde{X}^{\nu}_i\Big\} \Bigg]  dP^{\otimes N}(\mathbf{x},\mathbf{r})}_{=:B_N} \Bigg)^{\frac{1}{q}}. \label{ineq:UpperBound}
\end{align}

The first term of the right hand side of \eqref{ineq:UpperBound} can be controlled by large deviations estimates. Controlling the second term is the object of the following lemma. With these two results, we can now conclude as in \cite[Lemma 4.7]{ben-arous-guionnet:95}.

\end{proof}

\begin{lemma}\label{lemma3.1}
For any real number $q >1$, there exists a strictly positive real number $\delta_q>0$ and a function $C_q:]0,\delta_q[ \to \R$ such that $\lim_{\delta \rightarrow 0} C_q(\delta) =0$ and:
\begin{equation*}	
B_N \leq \exp\{C_q(\delta) N\}.
\end{equation*}
\end{lemma}

\begin{proof}
Using H\"older inequality with conjugate exponents $(\sigma, \eta)$, one finds:
\begin{equation}\label{ineqlemma}
B^N  \leq \underbrace{\E_{\gamma} \Bigg[ \int_{(\C \times D )^N} \exp\Big\{\sum_{i=1}^N \sigma \tilde{X}^{\nu}_i \Big\} dP^{\otimes N}(\mathbf{x},\mathbf{r}) \Bigg]^{\frac{1}{\sigma}}}_{=:(B^N_1)^{\frac{1}{\sigma}} }  \underbrace{ \E_{\gamma} \Bigg[ \int_{\hat{\mu}_N \in B(\nu,\delta)}  \prod_{i=1}^N \exp\Big\{ q \eta \big( \tilde{X}^{\hat{\mu}_N}_i -\tilde{X}^{\nu}_i \big) \Big\} dP^{\otimes N}(\mathbf{x},\mathbf{r}) \Bigg]^{\frac{1}{\eta}}}_{=:\big(B^N_2\big)^{\frac{1}{\eta}}}.
\end{equation}
The first term is controlled by martingale property:
\begin{multline*}
B^N_1 = \E_{\gamma} \Bigg[ \int_{D^N} \exp\bigg\{\sum_{i=1}^N \frac{\sigma^2-\sigma}{2} \int_0^T \big( \tilde{G}^{\nu}_t(r_i) + m_{\nu}(t,r_i) \big)^2 dt\bigg\} \\
\times  \int_{\C^N} \exp\bigg\{\sum_{i=1}^N \sigma\int_0^T \big( \tilde{G}^{\nu}_t(r_i) + m_{\nu}(t,r_i) \big) dW_t(x^i,r_i) - \frac{\sigma^2}{2} \int_0^T \big( \tilde{G}^{\nu}_t(r_i) + m_{\nu}(t,r_i) \big)^2 dt\bigg\} dP_{\mathbf{r}}(\mathbf{x}) d\pi^{\otimes N}(\mathbf{r})  \Bigg]\\
\overset{\text{Jensen,Fubini}}{\leq} \int_{D^N} \prod_{i=1}^N \Bigg\{ \int_0^T \E_{\gamma} \Bigg[  \exp\bigg\{ \frac{\sigma(\sigma-1)T}{2} \big( \tilde{G}^{\nu}_t(r_i) + m_{\nu}(t,r_i) \big)^2 \bigg\}\Bigg] \frac{dt}{T}\Bigg\} d\pi^{\otimes N}(\mathbf{r}) \overset{\eqref{eq:GaussianExponentialQuadraticMoment}}{\leq} \exp\Big\{ c_T (\sigma-1) N \Big\},
\end{multline*}
with $c_T$ uniform in space.

Let us control the second term, denoting $\kappa=q \eta$ and supposing that $\delta$ is small enough. By Cauchy-Schwarz's inequality and Fubini's  theorem:
\begin{align*}
& B^N_2 \leq \E_{\gamma} \Bigg[ \int_{(\C \times D)^N} \prod_{i=1}^N  \exp\Big\{ 2 \kappa \int_0^T \Big( \tilde{G}^{\hat{\mu}_N}_t(r_i)-\tilde{G}^{\nu}_t(r_i)+ \big(m_{\hat{\mu}_N}(t,r_i)-m_{\nu}(t,r_i)\big) \Big)dW_t(x^i,r_i)  \\
& - 2\kappa^2 \int_0^T \Big(\tilde{G}^{\hat{\mu}_N}_t(r_i) - \tilde{G}^{\nu}_t(r_i) + \big(m_{\hat{\mu}_N}(t,r_i) -m_{\nu}(t,r_i)\big)\Big)^2 dt \Big\} dP^{\otimes N}(\mathbf{x},\mathbf{r}) \Bigg]^{\frac{1}{2}} \\
& \times  \Bigg\{ \! \int_{\hat{\mu}_N \in B(\nu,\delta)} \! \E_{\gamma} \bigg[ \prod_{i=1}^N \exp\Big\{ 2\kappa^2 \! \int_0^T \! \Big(\tilde{G}^{\hat{\mu}_N}_t(r_i) \!- \!\tilde{G}^{\nu}_t(r_i)\! + \!\big(m_{\hat{\mu}_N}(t,r_i)\! -  \!m_{\nu}(t,r_i)\big)\Big)^2 \!dt\\
& - \kappa \int_0^T \Big(\tilde{G}^{\hat{\mu}_N}_t(r_i)+m_{\hat{\mu}_N}(t,r_i)\Big)^2 - \Big(\tilde{G}^{\nu}_t(r_i) + m_{\nu}(t,r_i)\Big)^2 dt \Big\} \bigg]  dP^{\otimes N}(\mathbf{x},\mathbf{r}) \Bigg\}^{\frac{1}{2}}
\end{align*}
The first term is equal to one by martingale property. For the second term, we remark that:
\begin{multline*}
- \int_0^T \Big(\tilde{G}^{\hat{\mu}_N}_t(r_i)+m_{\hat{\mu}_N}(t,r_i)\Big)^2 - \Big(\tilde{G}^{\nu}_t(r_i) + m_{\nu}(t,r_i)\Big)^2 dt \leq \\
\frac{\delta^{\frac{1}{2}}}{2} \bigg( \frac{1}{\delta} \int_0^T \Big(\tilde{G}^{\hat{\mu}_N}_t(r_i)- \tilde{G}^{\nu}_t(r_i) + \big(m_{\hat{\mu}_N}(t,r_i)  - m_{\nu}(t,r_i)\big)\Big)^2 dt \\
 + \int_0^T \Big(\tilde{G}^{\hat{\mu}_N}_t(r_i)+ \tilde{G}^{\nu}_t(r_i) + \big(m_{\hat{\mu}_N}(t,r_i)+m_{\nu}(t,r_i)\big)\Big)^2 dt \bigg)
\end{multline*}
 so that, by Cauchy-Schwarz's inequality:
 \begin{align*}
B^N_2 & \leq \Bigg\{ \int_{\hat{\mu}_N  \in B(\nu,\delta)} \! \! \! \! \! \!  \E_{\gamma} \bigg[  \prod_{i=1}^N \exp\Big\{ \big(4\kappa^2+\kappa\delta^{-\frac{1}{2}}\big) \int_0^T \big(\tilde{G}^{\hat{\mu}_N}_t(r_i)-\tilde{G}^{\nu}_t(r_i)+ (m_{\hat{\mu}_N}-m_{\nu})(t,r_i)\big)^2 dt \Big\} \bigg] dP^{\otimes N}(\mathbf{x},\mathbf{r}) \Bigg\}^{\frac{1}{4}} \\
& \times \Bigg\{ \int_{\big(\C \times D \big)^N} \prod_{i=1}^N \underbrace{ \E_{\gamma} \bigg[ \exp\Big\{ \kappa\delta^{\frac{1}{2}}\int_0^T \big(\tilde{G}^{\hat{\mu}_N}_t(r_i)+\tilde{G}^{\nu}_t(r_i)+ (m_{\hat{\mu}_N}+m_{\nu})(t,r_i)\big)^2 dt \Big\} \bigg]}_{ \overset{\eqref{eq:GaussianExponentialQuadraticMoment}}{\leq} \exp\Big\{ c_T \kappa\delta^{\frac{1}{2}} \Big\}  } dP^{\otimes N}(\mathbf{x},\mathbf{r})  \Bigg\}^{\frac{1}{4}}.
\end{align*}

Let us control the first term of the product, by taking advantage of the fact that ${\hat{\mu}_N \in B(\nu,\delta)}$. We have, for any $\xi \in \M_1^+\big((\C \times D)^2 \big)$ with marginals $\hat{\mu}_N$ and $\nu$:
\begin{align*}
& \big|m_{\hat{\mu}_N}-m_{\nu}\big|(t,r_i) \overset{\eqref{ineq:MajDiff}}{\leq} C_T d_T^V (\hat{\mu}_N,\nu) \leq C_T \delta,
\end{align*}
and similarly
\begin{align*}
\E_{\gamma} \Big[ & \big(\tilde{G}^{\hat{\mu}_N}_t(r_i)-\tilde{G}^{\nu}_t(r_i)\big)^2 \Big] \leq C \delta^2.
\end{align*}
Moreover, Jensen's inequality gives
\begin{align*}
\Big(\tilde{G}^{\hat{\mu}_N}_t(r_i)-\tilde{G}^{\nu}_t(r_i) + (m_{\hat{\mu}_N}-m_{\nu})(t,r_i)\Big)^2  \leq C \delta^2  + 2\big(\tilde{G}^{\hat{\mu}_N}_t(r_i)-\tilde{G}^{\nu}_t(r_i)\big)^2,
\end{align*}
so that by independence of the $\tilde{G}$ for different locations and \eqref{eq:GaussianExponentialQuadraticMoment}
\begin{align*}
\E_{\gamma} \bigg[ \prod_{i=1}^N \exp\Big\{ \big(4\kappa^2+\kappa\delta^{-\frac{1}{2}}\big) \int_0^T \big(\tilde{G}^{\hat{\mu}_N}_t(r_i)-\tilde{G}^{\nu}_t(r_i)+ (m_{\hat{\mu}_N}-m_{\nu})(t,r_i)\big)^2 dt \Big\} \bigg] \leq \exp\big\{ C_T \big(4\kappa^2+\kappa\delta^{-\frac{1}{2}}\big)\delta^2 N  \big\}.
\end{align*}
Hence, 
\begin{align*}
B^N_2 \leq  \exp\big\{ C_{\kappa}(\delta) N\big\}
\end{align*}
with $C_{\kappa}(\delta) \to 0$ as $\delta \to 0$.
\end{proof}

\begin{theorem}[Tightness]\label{thm:tightness}
For any real number $\varepsilon > 0$, there exists a compact set $K_{\varepsilon}$ of $\M_1^+(\C \times D )$ such that, for any integer $N$,
\begin{equation*}
Q^N(\hat{\mu}_N \notin K_{\varepsilon}) \leq \varepsilon.
\end{equation*}
\end{theorem}

\begin{proof}
The proof of this theorem consists in using the relative entropy inequality~\eqref{eq:IneqRelativeEntropy} and the exponential tightness of the sequence $(P^{\otimes N})_N$. The reader shall refer to \cite[Theorem 2]{cabana-touboul:12}, and remark that
\begin{align*}
\der{Q^N}{P^{\otimes N}}(\mathbf{x},\mathbf{r}) & \overset{\eqref{lemma2}}{=} \prod_{i=1}^N \E_{\gamma} \Big[ \exp\big\{ X^{\hat{\mu}_N}(x^i,r_i)\big\}\Big]\\
& \overset{\eqref{lemma5.15}}{=} \prod_{i=1}^N \exp\bigg\{\int_0^T O_{\hat{\mu}_N}(t,x^i,r_i)dW_t(x^i,r_i)-\frac{1}{2}\int_0^T O_{\hat{\mu}_N}(t,x^i,r_i)^2 dt\bigg\},
\end{align*}
where
\begin{align*}
O_{\hat{\mu}_N}(t,x,r)& := \E_{\gamma} \bigg[ \Lambda_t\big(G^{\hat{\mu}_N}(r)\big) G^{\hat{\mu}_N}_t(r) L^{\hat{\mu}_N}_t(x,r) \big) \bigg] + m_{\hat{\mu}_N}(t,r).
\end{align*}

to obtain inequality:
\begin{align}
& I(Q^{N}|P^{\otimes N}) = N  \int_{(\C \times D)^N} \bigg\{\int_0^T O_{\hat{\mu}_N}(t,x^1,r_1) dW_t(x^1,r_1) -\frac{1}{2} \int_0^T O_{\hat{\mu}_N}(t,x^1,r_1)^2 dt \bigg\} dQ^N(\mathbf{x},\mathbf{r}) \nonumber \\
& \overset{\text{Fubini}}{\leq} N \bigg\{ \int_{D^N} \int_0^T \underbrace{\int_{\C^N}  \E_{\gamma} \bigg[ \Lambda_t\big(G^{\hat{\mu}_N}(r_1)\big) G^{\hat{\mu}_N}_t(r_1) L^{\hat{\mu}_N}_t(x^1,r_1) \big) \bigg]^2 dQ^{N}_{\mathbf{r}}(\mathbf{x})}_{\varphi(t,\mathbf{r})} dt d\pi^{\otimes N}(\mathbf{r}) + \frac{\Jmax^2 T}{\ls^2} \bigg\}.
\end{align}

We then bound $\varphi(t,\mathbf{r})$ uniformly in space to conclude:
\begin{equation*}
\sup_{t \leq T} \varphi(t,\mathbf{r}) \leq  2\frac{\smax^4 T}{\ls^4}\exp{\Big\{2\frac{\smax^4 T}{\ls^4 } \Big\} }.
\end{equation*}

\end{proof}

\section{Identification of the mean-field equations}\label{sec:LimitIdentification}
We have seen that the series of empirical measures $\big(\hat{\mu}_N\big)_N$ satisfies a large deviations principle of speed $N$, and with good rate function $H$. In order to identify the limit of the system, we study in this section the minima of the functions $H$, and characterize them through an implicit equation. In the spin-glass model investigated in~\cite{ben-arous-guionnet:95}, existence and uniqueness of solutions was made difficult by the fact that the drift was not considered Lipschitz continuous. Moreover, the characterization of the possible minima of the good rate function $H$ was achieved through an intricate variational study. Here, we propose another approach that substantially simplifies this characterization. Moreover, because of the regularity of our dynamics, we propose an original contraction argument to show that the good rate function $H$ admits a unique minimum, proof that was yet to be developed in the context of the neuronal equations${}^1$ \footnote{${}^1$ For instance, in the non-spatialized case treated in~\cite{cabana-touboul:12} was used a strong assumption of linearity of the intrinsic dynamics (our function $f$) which implied that solutions were Gaussian, special case for which moment methods were used (see \cite{faugeras-touboul-etal:09}).}.

\begin{lemma}\label{Qcharac}
	Let $\mu$ be a probability measure on $\C \times D$ which minimizes $H$. Then
\begin{equation}\label{eq:densityMinimum}
	\mu \simeq P, \qquad \mu = Q_{\mu},
\end{equation}
where $\mu \to Q_{\mu}$ introduced in \eqref{def:Qnu} is well-defined from $\M_1^+(\C \times D) \to \M_1^+(\C \times D)$.
\end{lemma}

\begin{proof}
Let $\mu \in \M_1^+(\C \times D)$ that minimizes $H$, and define the probability measure $Q_{\mu} \in \M_1^+(\C \times D)$ as in \eqref{def:Qnu}:
\[
\forall (x,r) \in \C \times D, \quad \der{Q_{\mu}}{P}(x,r):= \E_{\gamma}\Big[ \exp\left\{ X^{\mu}(x,r) \right\} \Big],
\]
which is equivalent to $P$ by Theorem~\ref{thm:Qnu}.
As $H$ is a good rate function, its minimal value must be $0$, so that $H(\mu)=I(\mu|P)-\Gamma(\mu)=0$. This imply by Proposition~\ref{prop:GammaBehaviour} that $I(\mu|P)=\Gamma(\mu)< +\infty$, which in turn implies $\mu \ll P$. Theorem~\ref{thm:Qnu} ensures that $\forall \nu \in \M_1^+(\C \times D)$, $H_{\mu}(\nu)=I(\nu|Q_{\mu})=I(\nu|P)-\Gamma_{\mu}(\nu)$. In particular,
\[
I(\mu|Q_{\mu})= H_{\mu}(\mu)=I(\mu|P)-\Gamma_{\mu}(\mu)=I(\mu|P)-\Gamma(\mu)=H(\mu)=0,
\]
so that $\mu=Q_{\mu}$ ${}^2$ \footnote{${}^2$ For the properties of the relative entropy, see \cite{deuschel-stroock:89}.}. Furthermore $Q_{\mu} \simeq P$ is a consequence of $\bar{Q}_{\mu} \simeq P$, $I(Q_{\mu}|\bar{Q}_{\mu})<+{\infty}$ and $\der{Q_{\mu}}{\bar{Q}_{\mu}}>0$.
\end{proof}

We now prove that there exists a unique probability measure satisfying~\eqref{eq:densityMinimum}.

\begin{theorem}\label{thm:UniquenessOfTheMinimum} 
$\mu \to Q_{\mu}$ admits a unique fixed point on $\M_1^+(\C \times D)$.
\end{theorem}

\begin{proof}
Lemma \eqref{lemma5.15} gives
\begin{align*}
\der{Q_{\mu}}{P}(x,r) = \exp\bigg\{\int_0^T O_{\mu}(t,x,r)dW_t(x,r)-\frac{1}{2}\int_0^T O_{\mu}(t,x,r)^2dt\bigg\},
\end{align*}
where 
\begin{align*}
O_{\mu}(t,x,r)& := \E_{\gamma} \Big[ \Lambda_t\big( G^{\mu}(r) \big) G^{\mu}_t(r) L^{\mu}_t(x,r) \Big]+ m_{\mu}(t,r).
\end{align*}
Let $\mu \in \M_1^+\big( \C \times D \big), r \in D$, and recall that $dQ_{\mu}(x,r):= dQ_{\mu,r}(x)d\pi(r)$. By Girsanov's theorem $Q_{\mu,r}$ is the law of $\big(x^{\mu}_t(r)\big)_{t \in [0,T]}$, the unique strong solution of the SDE (see lemma~\ref{lem:ExistenceAndUnicityTilde})
\begin{equation}\label{SDE:uniqueness}
\left\{
  \begin{array}{ll}
  dx_t^{{\mu}}(r) = f(r,t,x_t^{{\mu}}(r)) dt + O^{\tilde{W}}_{\mu}(t,r) dt + \lambda(r) d\tilde{W}_t\\
    \big(x^{{\mu}}_s(r)\big)_{-\bar{\tau} \leq s \leq 0} = \big(\bar{x}^0_s(r)\big)_{-\bar{\tau} \leq s \leq 0}.
  \end{array}
\right.
\end{equation}
where $\tilde{W}$ is a $\mathbbm{P}$-Brownian motion,
\[
O^{\tilde{W}}_{\mu}(t,r):= \lambda(r) \E_{\gamma} \Big[ \Lambda_t\big( G^{\mu}(r) \big) G^{\mu}_t(r) \tilde{L}^{\mu}_t(r) \Big]+\lambda(r) m_{\mu}(t,r),
\]
\[
\tilde{L}^{\mu}_t(r):= \int_0^t G^{\mu}_s(r) \Big(d\tilde{W}_s-  m_{\mu}(s,r)ds\Big),
\]
and $\bar{x}^0(r) \in \mathcal{C}_{\tau}$ is the version of $\mu_0(r)$ of hypothesis \eqref{hyp:spaceRegInitCond}. It is easy to see that this process has a bounded second moment under the current assumptions. Let also $\nu \in \M_1^+\big( \C \times D \big)$, and let $x^{\nu}_t(r)$ be the unique strong solution of:
\begin{equation*}
\left\{
  \begin{array}{ll}
  dx_t^{{\nu}}(r) = f(r,t,x_t^{{\nu}}(r)) dt + O^{\tilde{W}}_{\nu}(t,r) dt + \lambda(r) d\tilde{W}_t\\
    \big(x^{{\nu}}_s(r)\big)_{-\bar{\tau} \leq s \leq 0} = \big(\bar{x}^0_s(r)\big)_{-\bar{\tau} \leq s \leq 0},
  \end{array}
\right.
\end{equation*}
where both the initial condition $\bar{x}^0(r)$ and the driving Brownian motion $(\tilde{W}_{t})$ are the same as for the definition of \eqref{SDE:uniqueness}. We have
\begin{align}
& \big( x^{\mu}_t(r) - x^{\nu}_t(r) \big) = \int_0^t \big(f(r,s,x^{\mu}_{s}(r))- f(r,s,x^{\nu}_s(r))+ \lambda(r) ( m_{\mu}(s,r)- m_{\nu}(s,r)\big) ds \nonumber\\
& + \lambda(r) \int_0^t \E_{\gamma} \bigg[ \Lambda_s\big( G^{\mu}(r) \big) G^{\mu}_s(r) \tilde{L}^{\mu}_s(r) - \Lambda_s\big( G^{\nu}(r) \big) G^{\nu}_s(r) \tilde{L}^{\nu}_s(r) \bigg] ds \label{ineq:FixPointAverage}.
\end{align}
As in proposition ~\ref{prop:MeanAndVarRegularity} we can obtain:
\begin{align*}
\lambda(r) (m_{\mu}(s,r)- m_{\nu}(s,r)) & \leq C_T d_s^V(\mu,\nu).
\end{align*}
Let $\xi \in \M_1^+\big((\C \times D)^2\big)$ with marginals $\mu$ and $\nu$, and let $\big(G,G'\big)$ be a bi-dimensional centered Gaussian process on the probability space $\big(\hat{\Omega}, \hat{\mathcal{F}}, \gamma \big)$ with covariance $K_{\xi}\big(\cdot, \cdot,r\big)$ given in \eqref{kxi}. Observe that
\begin{align}
& \E_{\gamma} \bigg[ \Lambda_t\big( G^{\mu}(r) \big) G^{\mu}_t(r) \tilde{L}^{\mu}_t(r) - \Lambda_t\big( G^{\nu}(r) \big) G^{\nu}_t(r) \tilde{L}^{\nu}_t(r) \bigg]= \E_{\gamma} \Big[ \Lambda_s(G) G_s L_s - \Lambda_s( G') G'_s L'_s \Big] \nonumber \\
& = \E_{\gamma} \Big[ \big(\Lambda_t(G) -\Lambda_t( G') \big) G_t L_t \Big] + \E_{\gamma} \Big[ \Lambda_t( G') \big(G_t -G'_t \big) L_t \Big] + \E_{\gamma} \Big[ \Lambda_t( G')  G'_t \big( L_t - L'_t\big) \Big]\nonumber \\
& \overset{\mathrm{C.S.}}{\leq} \E_{\gamma} \big[L_t^2 \big]^{\frac{1}{2}} \bigg( \E_{\gamma} \Big[ \big(\Lambda_t( G ) -\Lambda_t( G') \big)^2 G_t^2 \Big]^{\frac{1}{2}}+ \E_{\gamma} \Big[ \Lambda_t( G')^2 \big(G_t -G'_t\big)^2 \Big]^{\frac{1}{2}}\bigg) + \E_{\gamma} \Big[ \Lambda_t( G')^2  {G'_t}^2 \Big]^{\frac{1}{2}} \E_{\gamma} \Big[\big( L_t - L'_t\big)^2\Big]^{\frac{1}{2}} \label{ineq:uniquenesstheorem1}
\end{align}
where $L_t:= \int_0^t G_s d\big(d\tilde{W}_s- m_{\mu}(s,r)ds\big)$, and $L'_t:=\int_0^t G'_s d\big(d\tilde{W}_s- m_{\mu}(s,r)ds\big)$.

On the one hand, relying on \eqref{ineq:DiffLambda}, \eqref{ineq:BoundOnLambda} and Isserlis' theorem, we can show as in Proposition~\ref{prop:MeanAndVarRegularity} that there exists $C_T>0$ such that:
\begin{align*}
\E_{\gamma} \bigg[ \Big(\Lambda_t(G) -\Lambda_t(G') \Big)^2 G_t^2 \bigg]+ \E_{\gamma} \Big[ \Lambda_t(G')^2 \big(G_t -G'_t \big)^2 \Big] \leq C_T \left( \int_{(\C \times D)^{2}} d_t\big((y,r'),(z,\tilde{r})\big)^2 d\xi\big((y,r'),(z,\tilde{r})\big) \right).
\end{align*}
and
\[
\lambda(r) \E_{\gamma} \Big[ \Lambda_t(G')^2  {G'_t}^2 \Big]^{\frac{1}{2}} \leq C_T.
\]
On the other hand, remark that
\begin{align*}
& \E_{\gamma} \Big[\big( L_t - L'_t\big)^2\Big] \leq 2 \E_{\gamma} \bigg[\bigg(\int_0^t G_s-G'_s d\tilde{W}_s \bigg)^2\bigg]+ 2 \E_{\gamma} \bigg[ \bigg(\int_0^t G_s m_{\mu}(s,r) -G'_s m_{\nu}(s,r) ds \bigg)^2\bigg] \overset{\text{Jens., Fub.}}{\leq} \\
& 2 \E_{\gamma} \bigg[\bigg(\int_0^t G_s-G'_s d\tilde{W}_s \bigg)^2\bigg] +4t \int_0^t \Bigg\{ \E_{\gamma} \Big[ \big(G_s-G'_s\big)^2 m_{\mu}(s,r)^2 \Big] + \E_{\gamma} \Big[ {G'_s}^2 \big(m_{\mu}(s,r) - m_{\nu}(s,r)\big)^2  \Big]\Bigg\}ds\\
& \overset{\eqref{ineq:MajDiff}}{\leq} C_T \Bigg\{ \E_{\gamma} \bigg[\bigg(\int_0^t G_s-G'_s d\tilde{W}_s \bigg)^2\bigg] + \left( \int_{(\C \times D)^{2}} d_t\big((y,r'),(z,\tilde{r})\big)^2 d\xi\big((y,r'),(z,\tilde{r})\big) \right) \Bigg\},
\end{align*}
and also that
\[
\E_{\gamma} \big[L_t^2 \big] \overset{\mathrm{C.S.}}{\leq} 2 \E_{\gamma} \bigg[\bigg(\int_0^t G_s d\tilde{W}_s \bigg)^2\bigg]+ \underbrace{2t \int_0^t \E_{\gamma} \Big[ G_t^2 m_{\mu}(t,r)^2 \Big]}_{\leq C_T}.
\]
Injecting these result in \eqref{ineq:FixPointAverage}, we obtain:
\begin{align*}
& \big\Vert x^{\mu}(r) - x^{\nu}(r) \big\Vert_{\infty,t}^2 \leq C_T \int_0^t \Bigg\{ \big\Vert x^{\mu}(r)-x^{\nu}(r) \big\Vert_{\infty,s}^2 + \E_{\gamma} \bigg[ \sup_{v\leq s} \Big(\int_0^v G_u-G'_u d\tilde{W}_u \Big)^2\bigg] \\
& + \bigg( 1 + \E_{\gamma} \bigg[ \sup_{v\leq s} \Big(\int_0^v G_u d\tilde{W}_u \Big)^2\bigg]\bigg)\left( \int_{(\C \times D)^{2}} d_s\big((y,r'),(z,\tilde{r})\big)^2 d\xi\big((y,r'),(z,\tilde{r})\big) \right) \Bigg\} ds,
\end{align*}
so that by Gronwall's lemma
\begin{align*}
& d_t\Big((x^{\mu}(r),r),(x^{\nu}(r) ,r)\Big)^2 \leq C_T \int_0^t \Bigg\{ \E_{\gamma} \bigg[ \sup_{v\leq s} \Big(\int_0^v G_u-G'_u d\tilde{W}_u \Big)^2\bigg] \\
& + \bigg( 1 + \E_{\gamma} \bigg[ \sup_{v\leq s} \Big(\int_0^v G_u d\tilde{W}_u \Big)^2\bigg]\bigg)\left( \int_{(\C \times D)^{2}} d_s\big((y,r'),(z,\tilde{r})\big)^2 d\xi\big((y,r'),(z,\tilde{r})\big) \right)  \Bigg\} ds .
\end{align*}
Taking the expectation over the Brownian path and initial condition, and using Fubini's theorem and Burkholder-Davis-Gundy's inequality, we obtain
\begin{align}\label{ineq:uniqueness2}
\Exp \bigg[ d_t\Big((x^{\mu}(r),r),(x^{\nu}(r) ,r)\Big)^2 \bigg]& \leq C_T \int_0^t \left( \int_{(\C \times D)^{2}} d_s\big((y,r'),(z,\tilde{r})\big)^2 d\xi\big((y,r'),(z,\tilde{r})\big) \right)  ds.
\end{align}
We now show that we can integrate the term of the left-hand side over $\pi$.
To this purpose, chose distinct locations $r', r \in D$, and let $x^{\mu}_\cdot{(r')}$ be the strong solution of \eqref{SDE:uniqueness} with same $\tilde{W}$ but initial condition given by $\bar{x}^0(r')$ and intrinsic dynamics $f(r',\cdot,x^{\mu}_\cdot{(r')})$. Then
\begin{align*}
& \big| x^{\mu}_t(r) - x^{\mu}_t(r') \big| \leq  \big| \bar{x}^0_0(r) - \bar{x}^0_0(r') \big| + \big( K_f + K_{\lambda} \big|\tilde{W}_t\big|\big) \Vert r-r' \Vert_{\R^d} +  \int_0^t  \Big( K_f \big\Vert x^{\mu}(r) - x^{\mu}(r')\big\Vert_{\infty,s}  \\
& + C_T \big|\lambda(r) m_{\mu}(s,r)-\lambda(r')m_{\mu}(s,r')\big| \Big)ds \\
& +  \int_0^t \bigg|  \E_{\gamma} \bigg[ \lambda(r)\Lambda_t\big( G^{\mu}(r) \big) G^{\mu}_t(r) \tilde{L}^{\mu}_t(r) - \lambda(r') \Lambda_t\big( G^{\mu}(r') \big) G^{\mu}_t(r') \tilde{L}^{\mu}_t(r') \bigg] \bigg|ds,
\end{align*}
where the Gaussian process $(G^{\mu}(r),G^{\mu}(r'))$ has a covariance structure given by:
\begin{equation*}
\int_{\C \times D} 
\left(
\begin{array}{ccc}
\frac{\sigma(r, \tilde{r})^2}{\lambda(r)^2}  S(y_{s-\tau(r, \tilde{r})})S(y_{t-\tau(r, \tilde{r})})  & \frac{\sigma(r, \tilde{r})\sigma(r', \tilde{r})}{\lambda(r)\lambda(r')} S(y_{s-\tau(r, \tilde{r})})S(y_{t-\tau(r', \tilde{r})}) \\
\frac{\sigma(r, \tilde{r})\sigma(r', \tilde{r})}{\lambda(r)\lambda(r')} S(y_{s-\tau(r', \tilde{r})})S(y_{t-\tau(r, \tilde{r})}) & \frac{\sigma(r', \tilde{r})^2}{\lambda(r')^2} S(y_{s-\tau(r', \tilde{r})})S(y_{t-\tau(r', \tilde{r})}) \\
\end{array}
\right) 
d\mu \big(y,\tilde{r}\big).
\end{equation*}
First, observe that
\begin{align*}
& \big|\lambda(r) m_{\mu}(s,r) -\lambda(r')m_{\mu}(s,r') \big|= \Big|\int_{\C \times D} \Big(J(r,\tilde{r})S(y_{s-\tau(r,\tilde{r})})-J(r',\tilde{r})S(y_{s-\tau(r',\tilde{r})}) \Big) d\mu(y,\tilde{r})\Big|\\
& \leq K_J \Vert r-r'\Vert_{\R^d} + \Jmax K_S  \int_{\C \times D} \big|y_{s-\tau(r,\tilde{r})}-y_{s-\tau(r',\tilde{r})} \big| d\mu(y,\tilde{r}) \leq C \int_{\C \times D} d_s\big((y,r),(y,r')\big) d\mu(y,\tilde{r}).
\end{align*}
Moreover, as in Proposition~\ref{prop:MeanAndVarRegularity} we can show that:
\begin{multline*}
\E_{\gamma} \bigg[ \Big(\lambda(r) \Lambda_t\big( G^{\mu}(r) \big) -\lambda(r') \Lambda_t\big( G^{\mu}(r') \big) \Big)^2 G^{\mu}_t(r)^2 \bigg]+  \E_{\gamma} \Big[ \lambda(r')^2 \Lambda_t\big( G^{\mu}(r') \big)^2 \big( G^{\mu}_t(r) -  G^{\mu}_t(r') \big)^2 \Big]\\
\leq C_T \int_{\C \times D} d_t\big((y,r),(y,r')\big)^2 d\mu(y,\tilde{r}).
\end{multline*}
As a consequence, using the same decomposition as in inequality \eqref{ineq:uniquenesstheorem1}, we obtain
\begin{align*}
\big\Vert x^{\mu}(r) - x^{\mu}(r') \big\Vert_{\infty,t}^2 & \leq C_T \bigg\{ \Vert \bar{x}^0(r) - \bar{x}^0(r') \Vert_{\tau,\infty}^2+ \big( 1 + (\tilde{W}^*_t)^2\big) \Vert r-r' \Vert_{\R^d}^2+ \int_0^t \Vert x^{\mu}(r)-x^{\mu}(r') \Vert_{\infty,s}^2 ds \\
&+  \int_0^t \bigg( \int_{\C \times D} d_s\big((y,r),(y,r')\big)^2 d\mu(y,\tilde{r})\bigg) \bigg( 1 + \E_{\gamma} \bigg[ \sup_{v\leq s} \Big(\int_0^v G^{\mu}_u(r) d\tilde{W}_u \Big)^2\bigg]\bigg) ds \\
& + \int_0^t  \E_{\gamma} \bigg[ \sup_{v\leq s} \Big(\int_0^v G^{\mu}_u(r)-G^{\mu}_u(r') d\tilde{W}_u \Big)^2\bigg] ds \bigg\},
\end{align*}
where $\tilde{W}^*_t= \sup_{0 \leq s\leq t} |\tilde{W}_s|$. Using Gronwall's lemma, taking the expectation and relying again on Fubini's theorem and Burkholder-Davis-Gundy's inequality, we obtain:
\begin{align*}
\Exp \Big[ \Vert x^{\mu}(r) - x^{\mu}(r') \Vert_{\infty,t}^2 \Big] & \leq  C_T \bigg\{ \Exp \Big[ \Vert \bar{x}^0(r) - \bar{x}^0(r') \Vert_{\tau, \infty}^2\Big] + \big( 1 + \Exp \big[ (\tilde{W}^*_t)^2 \big]\big) \Vert r-r' \Vert_{\R^d}^2 \\
&  +  \int_0^t \int_{\C \times D} d_s\big((y,r),(y,r')\big)^2 d\mu(y,\tilde{r})  ds  \bigg\}.
\end{align*}
Hence $\Exp \Big[ \big\Vert x^{\mu}(r) - x^{\mu}(r') \big\Vert_{\infty,t}^2 \Big] \to 0$ as $\Vert r'- r \Vert_{\R^d} \searrow 0$, by using \eqref{hyp:spaceRegInitCond}, and the Monotone Convergence Theorem.
Now, observe that,
\begin{align*}
\underbrace{\Exp \Big[ d_t\Big((x^{\mu}(r),r),(x^{\nu}(r) ,r)\Big)^2 \Big]}_{=:\phi^{\mu,\nu}_t(r)}& = \Exp \Big[ \big\Vert x^{\mu}(r)-x^{\nu}(r) \big\Vert_{\infty,t}^2 \Big].
\end{align*}
As a consequence
\begin{align*}
& \big| \phi^{\mu,\nu}_t(r) - \phi^{\mu,\nu}_t(r') \big| = \Big| \Exp \Big[ \big\Vert x^{\mu}(r)-x^{\nu}(r) \big\Vert_{\infty,t}^2-\big\Vert x^{\mu}(r')-x^{\nu}(r') \big\Vert_{\infty,t}^2 \Big] \Big|,\\
 & \overset{\mathrm{C.S.}}{\leq} \prod_{\varepsilon= \pm 1} \Exp \bigg[ \Big( \big\Vert x^{\mu}(r)-x^{\nu}(r) \big\Vert_{\infty,t}+ \varepsilon \big\Vert x^{\mu}(r')-x^{\nu}(r') \big\Vert_{\infty,t} \Big)^2 \bigg]^{\frac{1}{2}}\\
 & \leq  \sqrt{2 \big(\phi^{\mu,\nu}_t(r')+ \phi^{\mu,\nu}_t(r)\big)} \Exp \bigg[ \Big( \big\Vert x^{\mu}(r)-x^{\mu}(r') \big\Vert_{\infty,t}+\big\Vert x^{\nu}(r)-x^{\nu}(r') \big\Vert_{\infty,t} \Big)^2\bigg]^{\frac{1}{2}}\\
 & \overset{\eqref{ineq:uniqueness2}}{\leq} C_T d_T^V(\mu,\nu) \sqrt{ \Exp \Big[ \big\Vert x^{\mu}(r)-x^{\mu}(r') \big\Vert_{\infty,t}^2\Big]+\Exp \Big[ \big\Vert x^{\nu}(r)-x^{\nu}(r') \big\Vert_{\infty,t}^2\Big]}.
\end{align*}
In particular, $r\to \phi^{\mu,\nu}_t(r)$ is continuous as soon as $d_T^V(\mu,\nu)< +\infty$ (which is for instance the case as soon as the two measures have a bounded second moment), and we can integrate inequality \eqref{ineq:uniqueness2} over space to obtain for any $\xi$ with marginals $\mu$ and $\nu$:
\begin{align}\label{ineq:uniqueness3}
\Exp \bigg[ \int_D  d_t\Big((x^{\mu}(r),r),(x^{\nu}(r) ,r)\Big)^2d\pi(r)  \bigg]  \leq C_T \int_0^t \left( \int_{(\C \times D)^{2}} d_s\big((y,r'),(z,\tilde{r})\big)^2 d\xi\big((y,r'),(z,\tilde{r})\big) \right)  ds.
\end{align}
The proof now relies on Picard's iteration. Define $Q_{\mu,\nu}$, the law of $\big((x^{\mu}(r),r),(x^{\nu}(r) ,r) \big)$ where $\mathcal{L}(r)=\pi$. It has marginal $Q_{\mu}$ and $Q_{\nu}$ so that for any $\xi \in \M_1^+((\C \times D)^2 \big)$ with marginals $\mu$ and $\nu$:
\begin{align*}
\Exp \bigg[ \int_D  d_t\Big((x^{Q_{\mu}}(r),r),(x^{Q_{\nu}}(r) ,r)\Big)^2d\pi(r)  \bigg]  & \overset{\eqref{ineq:uniqueness3}}{\leq} C_T \int_0^t \left( \int_{(\C \times D)^{2}} d_s\big((y,r'),(z,\tilde{r})\big)^2 dQ_{\mu,\nu}\big((y,r'),(z,\tilde{r})\big) \right)  ds\\
& \overset{\eqref{ineq:uniqueness3}}{\leq} C_T^2 \int_0^t \int_0^s \left( \int_{(\C \times D)^{2}} d_u\big((y,r'),(z,\tilde{r})\big)^2 d\xi\big((y,r'),(z,\tilde{r})\big) \right)  du  ds.
\end{align*}
Letting for any $\mu \in \M_1^+(\C \times D)$, $Q^0_{\mu}:=\mu$, and by recurrence $Q^n_{\mu}:=Q_{Q^{n-1}_{\mu}}$ for any $n \geq 1$, we can iterate this inequality to obtain for any $\xi \in \M_1^+((\C \times D)^2 \big)$ with marginals $\mu$ and $\nu$:
\begin{align*}
{d_t^V\big(Q^n_{\mu},Q^n_{\nu}\big)}^2 & \leq \Exp \bigg[ \int_D  d_t\Big((x^{Q^n_{\mu}}(r),r),(x^{Q^n_{\nu}}(r) ,r)\Big)^2d\pi(r)  \bigg] \\
& \overset{\eqref{ineq:uniqueness3}}{\leq} C_T^{n+1} \int_0^t \int_0^{t_1} \ldots \int_0^{t_n} \left( \int_{(\C \times D)^{2}} d_{t_{n+1}}\big((y,r'),(z,\tilde{r})\big)^2 d\xi\big((y,r'),(z,\tilde{r})\big) \right)  dt_{n+1} \ldots dt_1 ds\\
& \leq \frac{C_T^{n+1}t^{n+1}}{(n+1)!} \left( \int_{(\C \times D)^{2}} d_t\big((y,r'),(z,\tilde{r})\big)^2 d\xi\big((y,r'),(z,\tilde{r})\big) \right).
\end{align*}
Taking the infimum on $\xi$ now yields
\begin{align*}
{d_t^V\big(Q^n_{\mu},Q^n_{\nu}\big)}^2 \leq \frac{ C_T^{n+1} t^{n+1} }{ (n+1)! } d_t^V(\mu,\nu)^2.
\end{align*}
Hence, for $n$ large enough $\mu \to Q^n_{\mu}$ is a contraction so that we conclude the proof on classical arguments.
\end{proof}

\begin{lemma}\label{lem:ExistenceAndUnicityTilde}
For any $r\in D$ and $\mu\in \M_{1}^{+}(\C\times D)$, there exists a unique strong solution to the SDE:
\begin{equation*}
\left\{
  \begin{array}{ll}
  dx_t^{\mu}(r) = f(r,t,x_t^{\mu}(r)) dt + \lambda(r) O^{\tilde{W}}_{\mu}(t,r) dt + \lambda(r) d\tilde{W}_t \\
  \big(x_s^{\mu}(r)\big)_{-\bar{\tau} \leq s \leq 0} = \bar{x}^0(r).
  \end{array}
\right.
\end{equation*}
where $\tilde{W}$ is a $\mathbbm{P}$-Brownian motion, $\bar{x}^0(r) \in \C_{\tau}$ is the continuous realization of the initial law $\mu_0(r)$ of \eqref{hyp:spaceRegInitCond}, and
\[
O^{\tilde{W}}_{\mu}(t,r):= \E_{\gamma} \bigg[ \Lambda_t\big( G^{\mu}(r) \big)G^{\mu}_t (r)\bigg(\int_0^t G^{\mu}_s(r)\big(d\tilde{W}_s- m_{\mu}(s,r)ds\big)\bigg)\bigg]   + m_{\mu}(t,r).
\]
\end{lemma}

\begin{proof}
The proof relies on Picard's iterations. Let $x^0 \in \C_{\tau}$ with $x^0=\bar{x}^0(r)$, and define recursively the sequence $\big(x^n_t, 0\leq t \leq T \big)_{n \in \N^*}$ by $(x^n_s)_{ -\bar{\tau} \leq s \leq 0} = \bar{x}^0(r)$, and
\[
x^{n+1}_t = \bar{x}^0_0(r) + \int_0^t f(r,s,x_s^{n}) ds + \int_0^t \lambda(r) O^{\tilde{W}}_{\mu}(s,r) ds + \lambda(r) \tilde{W}_t, \; \; \forall t \in [0,T].
\]
We have for any $n \geq 1, t \in [0,T]$
\[
x^{n+1}_t -x^n_t = \int_0^t \big(f(r,s,x_s^{n})-f(r,s,x_s^{n-1})\big) ds,
\]
so that we easily obtain by Lipschitz-continuity of $f$ that
\[
\Exp \Big[ \sup_{s \leq t} \big| x^{n+1}_s- x^n_s \big|^2 \Big] \leq C_T \int_0^t \Exp \Big[ \sup_{u \leq s} \big| x^{n}_u- x^{n-1}_u \big|^2 \Big] ds.
\]
The conclusion relies on classical arguments.
\end{proof}

\subsection{Convergence of the process}

We are now in a position to prove theorem~\ref{thm:Convergence}.

\begin{proof}
	Indeed, for $\delta$ a strictly positive real number and $B(Q,\delta)$ the open ball of radius $\delta$ centered in $Q$ for the Vasserstein distance. We prove that $Q^{N}(\hat{\mu}_N\notin B(Q,\delta))$ tends to zero as $N$ goes to infinity. Indeed, for $K_{\varepsilon}$ a compact defined in theorem~\ref{thm:tightness}, we have for any $\varepsilon>0$:
		\begin{equation*}
	Q^N\big(\hat{\mu}_N \notin B(Q,\delta)\big) \leq \varepsilon + Q^N\big(\hat{\mu}_N \in B(Q,\delta)^c \cap K_{\varepsilon}\big).
	\end{equation*}
	The set $B(Q,\delta)^c \cap K_{\varepsilon}$ is a compact, and theorem~\ref{lemma3} now ensures that
	\[\limsup_{N\to\infty} \frac 1 N \log Q^{N}(\hat{\mu}_N\in B(Q,\delta)^c \cap K_{\varepsilon} )\leq -\inf_{B(Q,\delta)^c \cap K_{\varepsilon}} H\]
	and eventually, theorem~\ref{thm:Limit} ensures that the right-hand side of the inequality is strictly negative, which implies that
	\[\lim_{N\to\infty}Q^{N}(\hat{\mu}_N \notin B(Q,\delta)) \leq \varepsilon,\]
	that is:
	\[\lim_{N \to\infty}Q^{N}(\hat{\mu}_N \notin B(Q,\delta)) =0.\]
\end{proof}

\appendix

\section{Proof of lemma~\ref{lemma:ContinuityOfSolutions}: regularity of the solutions for the limit equation}\label{sec:Appendix.A}
In this appendix we demonstrate the regularity in space of the solutions that is expressed in lemma~\ref{lemma:ContinuityOfSolutions}. We start by showing a technical lemma on the uncoupled system before proceeding to the proof of that result. 
\begin{lemma}
\begin{enumerate}
\item The map:
\[
\mathcal{P}:
\left\{
  \begin{array}{ll}
  D \to \M_1^+(\C)  \\
  r \to P_r\\
  \end{array}
  \right.
\]
is continuous with respect to the borel topology on $D$, and the weak topology on $\M_1^+(\C)$, e.g. $r_n \to r \quad \implies P_{r_n} \overset{\mathcal{L}}{\to} P_r$.
\item Let $\mathcal{W}$ be the Wiener measure on $\C$. Then, $\forall A \in \mathcal{B}(\C), \mathcal{W}(A)=0 \implies P_r(A)=0$. 
\item $P$ is a well defined probability measure on $\C \times D$.
\end{enumerate}
\end{lemma}

\begin{proof}
The first point is the consequence of a coupling argument. Let $W$ be a $P$-Brownian motion, $\bar{x}^0:D \to \C_{\tau}$ be as in \eqref{hyp:spaceRegInitCond}, and $(r_n)_{n \in \N} \in D^{\N}$ a sequence of positions that converges toward $r \in D$. We consider $X^n$ and $X$, the respective strong solutions of the SDEs:
\[
  \begin{cases}
      dX^n_t= f(r_n,t,X^n_t) dt + \lambda(r_n) dW_t\\
      (X^n_t)_{t \in [-\bar{\tau},0] } = \bar{x}^0(r_n)
  \end{cases} \qquad \qquad
  \begin{cases}
      dX_t= f(r,t,X_t) dt + \lambda(r) dW_t\\
      (X_t)_{t \in [-\bar{\tau},0] } = \bar{x}^0(r)
  \end{cases}
\]

driven by the same Brownian motion $W$.\\
Then, by Gronwall lemma, letting $W^*_T= \underset{t\in [-\bar{\tau},T]}{\sup} |W_t|$,
\[
  \Vert X^n-X \Vert_{\infty,T} \leq  \Big( \Vert\bar{x}^0(r_n)-\bar{x}^0(r) \Vert_{\tau,\infty} + \Vert r-r_n \Vert_{\R^d}  K_f T+ K_{\lambda}\Vert r-r_n \Vert_{\R^d}   W^*_T \Big) e^{\{ K_f T \}} ,
\]
Hence, by \eqref{hyp:spaceRegInitCond}:
\begin{align*}
\Exp \Big[ \Vert X^n-X \Vert_{\infty,T}^2 \Big] \to 0, \text{ as } r_n \to r,
\end{align*}
so that $P_{r_n}=\mathcal{L}(X_n) \implies \mathcal{L}(X)=P_r$ as $r_n \to r$.\\
In order to prove the second point, let $W_r$ be the unique strong solution of
\[
  \begin{cases}
      dX_t= \lambda(r) dW_t\\
      (X_t)_{t \in [-\bar{\tau},0] } = \bar{x}^0(r).
  \end{cases}
\]
Following Exercise (2.10) of \cite{revuz-yor:99}, we remark, by Lipschitz continuity of $f$, that explosion of $P_r$ almost surely never occurs in finite time, so that Girsanov's theorem applies:
\[
P_r \ll W_r, \quad \quad \der{P_r}{W_r}=\exp{ \Big\{ \int_0^T \frac{f(r,t,X_t)}{\lambda(r)} dX_t - \frac 1 2 \int_0^T \Big(\frac{f(r,t,X_t)}{\lambda(r)}\Big)^2 dt \Big\}}.
\]
Consequently, $ \forall A \in \mathcal{B}(\C)$,
\[
P_r(A)=\Exp_{W_r}(\der{P_r}{W_r} \mathbf{1_A})
\]
so that $P_r(A)=0$ as soon as $W_r(A)=0$. As $\lambda(r) > \ls$, $W_r(A)=0 \iff \mathcal{W}(A)=0$.\\
The third point is now easy to settle. In fact, for any $y \in \C$ and $\varepsilon > 0$, $\mathcal{W}(\partial \mathcal{B}(y,\varepsilon))=\mathcal{W}(\{x \in \C, \lVert x-y\rVert_{\infty,T}=\varepsilon\})=0$. Hence, Portmanteau implies that $r \to P_r(\mathcal{B}(y,\varepsilon))$ is a continuous map, so that we can define $\int_{D} P_r(\mathcal{B}(y,\varepsilon)) d\pi(r)$ univocally. As $\{ \mathcal{B}(y,\varepsilon)\times B, y \in \C, \varepsilon > 0, B\in \mathcal{B}(D)\}$ form a $\Pi$-system that generates $\mathcal{B}(\C \times D)$, $P$ is a well defined probability measure on $\C \times D$.
\end{proof}

We now proceed to prove lemma \ref{lemma:ContinuityOfSolutions} that we repeat below:
\begin{lemma}
The map
\begin{equation*}
\mathcal{Q}:
\left\{
\begin{array}{ll}
  D^N \to \mathcal{M}_1^+(\C^N)\\
  \mathbf{r} \to Q^N_{\mathbf{r}}

\end{array}
\right.
\end{equation*}
where $Q^N_{\mathbf{r}}:=\mathcal{E}_J \big( Q^N_{\mathbf{r}}(J) \big)$, is continuous with respect to the weak topology. Moreover,
\[
dQ^N(\mathbf{x},\mathbf{r}):=dQ^N_{\mathbf{r}}(\mathbf{x})d\pi^{\otimes N}(\mathbf{r})
\]
defines a probability measure on $\mathcal{M}_1^+\big((\C \times D)^N \big)$.
\end{lemma}

\begin{remark}
$\mathcal{Q}$ maps the positions $\mathbf{r}$ to the Gaussian averaged of the solutions $Q^N_{\mathbf{r}}(J)$, so that its continuity seems to be a consequence of Cauchy-Lipschitz theorem with parameter $\mathbf{r}$. Yet, the equation depends on $\mathbf{r}$ through the Gaussian synaptic weights $J_{ij}$ which only satisfy a continuity in law. Meanwhile the proof is not difficult, it must rely on another argument. The one developed here is a coupling method.
\end{remark}

\begin{proof}
Before proceeding to it, remark that the Gaussian synaptic connections display some regularity in space. In fact, although the variables $J_{ij}$ and $J_{i'j'}$ are independent for $i \neq i'$ or $j \neq j'$, their probability distribution continuously depends on the spatial location of the cells, in the sense that one can find a version $\tilde{J}_{i'j'}$ of $J_{i'j'}$ such that:
\begin{equation}\label{eq:SynWeightsRegularity}
\mathcal{E}_J \Big( \big|\tilde{J}_{i'j'}- J_{ij}\big| \Big) \leq \frac{C}{N} \Big(\Vert r_i-r_{i'}\Vert_{\R^d} + \Vert r_j-r_{j'} \Vert_{\R^d} \Big).
\end{equation}
for some $C>0$ independent of the neurons locations. This is a consequence of the Lipschitz-continuity and boundedness of the mean and variance maps $J$ and $\sigma$.

We now proceed to the proof, and insist on the fact that $N$ remains constant. Fix a deterministic sequence $\Big(\mathbf{r}^n=(r^n_i)_{1 \leq i \leq N}\Big)_{n \in \N^*} \to_n \mathbf{r}=(r_i)_{1 \leq i \leq N} \in D^N$, let $(W^i_t, 0 \leq t \leq T)_{i\in [\![1, N ]\!]}$ be a family of independent $\P$-Brownian motions, and $\bar{x}^{0,i}:D \to \C_{\tau}, 1 \leq i \leq N$, be $N$ independent initial condition as in \eqref{hyp:spaceRegInitCond}. Let now $X^N_{\mathbf{r}^n}=\big(X^{i,N}_{\mathbf{r}^n}\big)_{i \in [\![1,N]\!]}$ and $X^N_{\mathbf{r}}=\big(X^{i,N}_{\mathbf{r}}\big)_{i \in [\![1,N]\!]}$ be the respective strong solutions of the two following stochastic differential equations:

\[
  \begin{cases}
      dX^{i,N}_{\mathbf{r}^n}(t)= \Big(f(r^n_i,t,X^{i,N}_{\mathbf{r}^n}(t))  + \sum_{j=1}^N \tilde{J}_{ij}^{\mathbf{r}^n} S\big(X^{j,N}_{\mathbf{r}^n}(t-\tau_{r^n_i r^n_j})\big) \Big)dt + \lambda(r^n_i) dW^i_t\\
      (X^N_{\mathbf{r}^n}(t))_{t \in [-\bar{\tau},0] } = \big( \bar{x}^{0,i}(r^n_i)\big)_{1\leq i \leq N},
  \end{cases}
\]

\[
  \begin{cases}
      dX^{i,N}_{\mathbf{r}}(t)= \Big(f(r_i,t,X^{i,N}_{\mathbf{r}}(t)) + \sum_{j=1}^N J_{ij}^{\mathbf{r}} S\big(X^{j,N}_{\mathbf{r}}(t-\tau_{r_i r_j})\big) \Big)dt + \lambda(r_i) dW^i_t\\
      (X^N_{\mathbf{r}}(t))_{t \in [-\bar{\tau},0] } = \big( \bar{x}^{0,i}(r_i)\big)_{1\leq i \leq N}.
  \end{cases} \qquad \qquad
\]
where $J^{\mathbf{r}}_{ij} \sim \mathcal{N} \Big( \frac{J(r_i,r_j)}{N}, \frac{\sigma(r_i,r_j)^2}{N} \Big)$, $\tilde{J}^{\mathbf{r}^n}_{ij} \sim J^{\mathbf{r}^n}_{ij}$ satisfy \eqref{eq:SynWeightsRegularity}, and where we used the short-hand notation $\tau_{r r'}:= \tau(r,r')$. In particular, $X^{i,N}_{\mathbf{r}^n}$ has law $Q^N_{\mathbf{r}^n}(J^{\mathbf{r}^n})$, and $X^{i,N}_{\mathbf{r}}$ has law $Q^N_{\mathbf{r}}(J^{\mathbf{r}})$.

Then, we have for every $t \in [0,T]$,
\begin{align*}
\big(X^{i,N}_{\mathbf{r}^n}(t)-X^{i,N}_{\mathbf{r}}(t)\big) & =  \big(\bar{x}^{0,i}_0(r^n_i)-\bar{x}^{0,i}_0(r_i) \big) +\Bigg( \int_0^t \Big(f(r^n_i,s,X^{i,N}_{\mathbf{r}^n}(s)) - f(r_i,s,X^{i,N}_{\mathbf{r}}(s))\Big)ds \\
& + \sum_{j=1}^N \Bigg\{ \big(\tilde{J}_{ij}^{\mathbf{r}^n} - J_{ij}^{\mathbf{r}}\big) \int_0^t S\big(X^{j,N}_{\mathbf{r}^n}(s-\tau_{r^n_i r^n_j})\big) ds \\
& + J_{ij}^{\mathbf{r}} \int_0^t \Big( S\big(X^{j,N}_{\mathbf{r}^n}(s-\tau_{r^n_i r^n_j})\big) - S\big( X^{j,N}_{\mathbf{r}}(s-\tau_{r_i r_j})  \big) \Big) ds \Bigg\} \Bigg) + \big(\lambda(r^n_i)-\lambda(r_i)\big) W^i_t.
\end{align*}
Let $W^{*,i}_T= \underset{t\in [0,T]}{\sup} |W^i_t|$. Then using Lipschitz continuity of $f$, $\lambda$, $S$, the fact that $|S| \leq 1$, and taking the supremum in time one obtains
\begin{align*}
& \big\Vert X^{i,N}_{\mathbf{r}^n}-X^{i,N}_{\mathbf{r}}\big\Vert_{\infty,t} \leq \Vert \bar{x}^{0,i}(r^n_i)-\bar{x}^{0,i}(r_i) \Vert_{\tau,\infty} +  \Big(  K_f t + K_{\lambda}  W^{*,i}_T \Big) \Vert r^n_i-r_i \Vert_{\R^d}  \\
& + \int_0^t \bigg\{ K_f \big\Vert X^{i,N}_{\mathbf{r}^n}-X^{i,N}_{\mathbf{r}}\big\Vert_{\infty,s} + K_S \sum_{j=1}^N  |J_{ij}^{\mathbf{r}}| \big\Vert X^{j,N}_{\mathbf{r}^n}-X^{j,N}_{\mathbf{r}}\big\Vert_{\infty,s} \bigg\} ds  + \sum_{j=1}^N \Bigg\{ t \big|\tilde{J}_{ij}^{\mathbf{r}^n} - J_{ij}^{\mathbf{r}}\big|\\
& + K_S  \int_0^t  |J_{ij}^{\mathbf{r}}| \underset{|a-b| \leq 2 K_{\tau} \Vert \mathbf{r}^n-\mathbf{r} \Vert_{\infty}}{\underset{a,b \in [-\bar{\tau},0]}{\sup}} \Big| X^{j,N}_{\mathbf{r}}(s+a) -  X^{j,N}_{\mathbf{r}}(s+b) \Big| ds \Bigg\}.
\end{align*}

where $\lVert \mathbf{r}\rVert_{\infty} = \sup_{1 \leq i \leq N} \Vert r_i \Vert_{\R^d}$. Let us denote, for any $X= \big( X^i \big)_{1 \leq i \leq N} \in \C^N$, $t \in [-\bar{\tau},T]$, $\lVert X \rVert^{1}_{\infty,t}= \sum_{1 \leq i \leq N} \Vert X^i \Vert_{\infty,t}$, and for any $\mathbf{r} \in D^N$, $\lVert \mathbf{r}\rVert_{1} = \sum_{1 \leq i \leq N} \Vert r_i \Vert_{\R^d}$. Summing over $i \in [\![1,N]\!]$ and using Gronwall's inequality now yields

\begin{align}
& \big\Vert X^{N}_{\mathbf{r}^n}-X^{N}_{\mathbf{r}}\big\Vert_{\infty,t}^1 \leq C_T \exp\Big\{ C_T \sum_{i,j=1}^N |J_{ij}^{\mathbf{r}}|\Big\} \Bigg( \Vert \bar{\mathbf{x}}^{0}(\mathbf{r}^n)-\bar{\mathbf{x}}^{0}(\mathbf{r}) \Vert_{\tau,\infty}^1 +  \big(\sup_{1\leq i \leq N}  W^{*,i}_T \big) \Vert \mathbf{r}^n-\mathbf{r} \Vert_1 \nonumber \\
& + \sum_{i,j=1}^N t \big|\tilde{J}_{ij}^{\mathbf{r}^n} - J_{ij}^{\mathbf{r}} \big| + \Big(\sum_{i,j=1}^N |J_{ij}^{\mathbf{r}}|\Big) \int_0^t \sup_{1 \leq j \leq N} \bigg\{ \underset{|a-b| \leq 2 K_{\tau} \Vert \mathbf{r}^n-\mathbf{r} \Vert_{\infty}}{\underset{a,b \in [-\bar{\tau},0]}{\sup}} \Big| X^{j,N}_{\mathbf{r}}(s+a) -  X^{j,N}_{\mathbf{r}}(s+b) \Big|  \bigg\} ds \Bigg) \label{ineq:ContinuitySpaceNetSolution1}.
\end{align}

The synaptic connections being Gaussian, observe that there exists a map $\chi:D^N \times D^N \to \R$ such that $\chi(\mathbf{r}^n, \mathbf{r}) \to 0$, when $n \to \infty$, and a constant $C_{T,N}$ such that:

\begin{align*}
&  \E_J \bigg[ \exp\Big\{ 2 C_T \sum_{i,j=1}^N |J_{ij}^{\mathbf{r}}|\Big\} \bigg]^{\frac{1}{2}} + \Exp \Big[  \big(\sup_{1\leq i \leq N}  W^{*,i}_T \big)^2 \Big]+ \E_{J} \bigg[\Big(\sum_{i,j=1}^N |J_{ij}^{\mathbf{r}}|\Big)^2 \bigg] \leq  C_{T,N} 
\end{align*}
\begin{align*} 
\Exp \bigg[ \Big(\Vert \bar{\mathbf{x}}^{0}(\mathbf{r}^n)-\bar{\mathbf{x}}^{0}(\mathbf{r}) \Vert_{\tau,\infty}^1\Big)^2 \bigg]  + \big(\Vert \mathbf{r}^n-\mathbf{r} \Vert_1\big)^2 + \sum_{i,j=1}^N  t^2 \E_J \Big[ \big(\tilde{J}_{ij}^{\mathbf{r}^n} - J_{ij}^{\mathbf{r}} \big)^2 \Big] \overset{\eqref{hyp:spaceRegInitCond}, \eqref{eq:SynWeightsRegularity}}{\leq} C_{T,N} \chi(\mathbf{r}^n, \mathbf{r}).
\end{align*}

Denoting $\Exp_J \big[ \cdot \big] := \Exp \big[ \E_J \big[ \cdot \big] \big]$, we find, taking the expectation in \eqref{ineq:ContinuitySpaceNetSolution1} and relying on Cauchy-Schwarz's inequality:
\begin{align*}
& \Exp_J \Big[ \Big[ \big\Vert X^{N}_{\mathbf{r}^n}-X^{N}_{\mathbf{r}}\big\Vert_{\infty,t}^1 \Big] \leq \tilde{C}_{T,N} \bigg( \chi(\mathbf{r}^n, \mathbf{r}) + \Exp_J \Bigg[  \sup_{1 \leq j \leq N} \underset{|a-b| \leq 2 K_{\tau} \Vert \mathbf{r}^n-\mathbf{r} \Vert_{\infty}}{\underset{a,b \in [-\bar{\tau},0],s \in [0,t]}{\sup}} \Big| X^{j,N}_{\mathbf{r}}(s+a) -  X^{j,N}_{\mathbf{r}}(s+b) \Big|^2 \Bigg] \Bigg)^{\frac{1}{2}}.
\end{align*}

As solution are $\P$-almost surely continuous, and $N$ remains (here) finite, the Monotone Convergence Theorem ensures that the right-hand side tends toward $0$ when $n$ goes to infinity. It implies in particular that $Q^N_{\mathbf{r}^n}$ converges in law toward $Q^N_{\mathbf{r}}$ when $n \to \infty$, so that the map $\mathbf{r} \to \mathcal{E}_J \Big[ \int_{\C^N} \phi(\mathbf{x})dQ^N_{\mathbf{r}}(\mathbf{x}) \Big]$ is continuous and integrable with respect to $\pi^{\otimes N}$. In particular, $dQ^N(\mathbf{x},\mathbf{r}):=dQ^N_{\mathbf{r}}(\mathbf{x}) d\pi^{\otimes N}(\mathbf{r})$ defines a probability measure on $(\C \times D)^N$.

\end{proof}

\end{document}